\documentclass[10pt,final,leqno,onefignum,onetabnum,hidelinks]{siamltex}
\usepackage{amssymb,amsfonts,amsmath}
\usepackage{hyperref}
\usepackage{graphicx}
\usepackage{color}
\usepackage{eepic,epic}
\usepackage{epsfig,subfigure,epstopdf}
\usepackage[showonlyrefs]{mathtools}

\newtheorem{example}{Example}[section]
\newtheorem{remark}{Remark}[section]
\definecolor{darkred}{rgb}{0.85,0,0}

\definecolor{green}{rgb}{0,0.7,0}

\newcommand{\al}{\alpha}

\def\II{(\Omega)}

\def\R{{\mathbb R}}
\def\d{{\mathrm d}}

\begin{document}

\title{Numerical analysis of nonlinear subdiffusion equations \thanks{The work of B. Jin is partially supported by UK EPSRC grant EP/M025160/1. 
The work of B. Li is partially supported by a grant from the Research Grants Council of the Hong Kong Special Administrative Region (Project No. 15300817). 
The work of Z. Zhou is partially supported by the AFOSR MURI center for Material Failure Prediction through peridynamics and the ARO MURI Grant W911NF-15-1-0562.} }
\author{Bangti Jin\thanks{Department of Computer Science, University College London, Gower Street, London, WC1E 6BT, UK
(\texttt{b.jin@ucl.ac.uk, bangti.jin@gmail.com})}
\and Buyang Li\thanks{Department of Applied Mathematics, The Hong Kong Polytechnic University, Kowloon, Hong Kong. 
(\texttt{buyang.li@polyu.edu.hk, libuyang@gmail.com})} 
\and Zhi Zhou\thanks{Department of Applied Mathematics, The Hong Kong Polytechnic University, Kowloon, Hong Kong.
%Department of Applied Physics and Applied Mathematics, Columbia University, 500 W. 120th Street, New York, NY 10027, USA 
(\texttt{zhizhou@polyu.edu.hk, zhizhou0125@gmail.com})}}

\date{\today}

\maketitle
\begin{abstract}
We present a general framework for the rigorous numerical analysis of time-fractional nonlinear parabolic partial differential equations, with a fractional derivative of order $\alpha\in(0,1)$ in time. It relies on three technical tools:
a fractional version of the discrete Gr\"onwall-type inequality, discrete maximal regularity, and regularity theory of nonlinear equations. We establish a general criterion for showing the fractional discrete Gr\"onwall inequality, and verify it for the L1 scheme and convolution quadrature generated by BDFs. Further, we provide a complete solution theory, e.g., existence, uniqueness and regularity, for a time-fractional diffusion equation with a Lipschitz nonlinear source term. Together with the known results of discrete maximal regularity, we derive pointwise $L^2(\Omega)$ norm error estimates for semidiscrete Galerkin finite element solutions and fully discrete solutions, which are of order $O(h^2)$ (up to a logarithmic factor) and $O(\tau^\alpha)$, respectively, without any extra
regularity assumption on the solution or compatibility condition on the problem data. The sharpness of the convergence rates is supported by the numerical experiments. \smallskip

\end{abstract}

\section{\bf Introduction}\label{Se:intr}

Time-fractional parabolic partial differential equations (PDEs)
have been very popular for modeling anomalously slow transport processes in the past two decades.
These models are commonly referred to as fractional diffusion or subdiffusion. At a microscopic
level, the underlying stochastic process is continuous time random walk \cite{MetzlerKlafter:2000}. So far they have been
successfully applied in a broad range of diversified research areas, e.g., thermal diffusion in
fractal domains \cite{Nigmatulin:1986}, flow in highly heterogeneous aquifer \cite{Berkowitz:2002}
and single-molecular protein dynamics \cite{Kou:2008}, just to name a few.
Hence, the rigorous numerical analysis of such problems is of great practical
importance. For the linear problem, various efficient time stepping schemes have been
proposed, which include mainly two classes: L1 type schemes and convolution quadrature (CQ).

L1 type schemes approximate the fractional derivative
by replacing the integrand with its piecewise polynomial interpolation \cite{LanglandsHenry:2005,LinXu:2007,SunWu:2006,Alikhanov:2015}
and thus generalize the classical finite difference method.
The piecewise linear case has a local truncation error $O(\tau^{2-\alpha})$ for sufficiently smooth solution, where $\tau$ denotes the time step size. See also \cite{McLeanMustapha:2015,MustaphaAbdallahFurati:2014}
for the discontinuous Galerkin method. CQ is a flexible framework introduced by Lubich
\cite{Lubich:1986,Lubich:1988} for constructing high-order time discretization methods for approximating
fractional derivatives. It approximates the fractional derivative in the Laplace
domain and automatically inherits the stability property of general linear multistep methods. See
\cite{CuestaLubichPalencia:2006,Yuste:2006,ZengLiLiuTurner:2015,JinLazarovZhou:SISC2016} for CQ type
schemes. Optimal error estimates have been derived for both spatially semidiscrete and fully discrete schemes,
including problems with nonsmooth data \cite{CuestaLubichPalencia:2006,JinLazarovZhou:SIAM2013,McLeanMustapha:2015,JinLazarovZhou:SISC2016}.

However, up to now, there has been very few work on the rigorous numerical analysis of nonlinear time fractional
diffusion equations.  In this paper, we present a general framework for analyzing discretization errors
of nonlinear problems. The error of the numerical solution can be split into a
linear part and a nonlinear part. While the linear part has been carefully studied, the analysis of the
nonlinear part requires different mathematical machineries, in order to derive sharp error estimates.
Besides regularity estimates for the nonlinear problem, it requires
discrete maximal $\ell^p$ regularity, and a fractional version of the discrete
Gr\"onwall's inequality for time stepping schemes. The former gives a bound on the
discrete fractional derivative due to the nonlinear part, whereas the latter allows combining the
nonlinear part with the linear part to obtain a global error estimate.

To the best of our knowledge, a fractional version of discrete Gr\"onwall's inequality for
time stepping schemes is still unavailable in the literature. We shall establish such
discrete Gr\"onwall's inequality for both L1 scheme and CQs generated
by backward difference formulas (BDFs) up to order 6 in Theorem \ref{L1-CQ-gronwall}. Further,
in Theorem \ref{Other-Scheme-A2}, we present a general criterion under which the fractional
discrete Gr\"onwall's inequality holds. %Such a general criterion is useful for analyzing other
%time-stepping schemes in the future. %, e.g., high-order methods with initial corrections.

To illustrate the main idea of {this framework,}
we consider the following nonlinear problem in a bounded convex polygonal domain $\Omega\subset\R^d$, $d\ge 1$:
\begin{align}\label{nonlinear-PDE}
\left\{
\begin{aligned}
&\partial_t^\alpha u-\Delta u=f(u) &&\mbox{in}\,\,\,\Omega\times(0,T),\\
&u=0 &&\mbox{on}\,\,\partial\Omega\times(0,T),\\
&u=u_0 &&\mbox{in}\,\, \Omega\times\{0\},
\end{aligned}
\right.
\end{align}
where $u_0\in H_0^1(\Omega)\cap H^2(\Omega)$ is a given function
and $f:\mathbb{R}\rightarrow \mathbb{R}$ is a Lipschitz continuous function, i.e.,
$|f(s)-f(t)|\leq L|s-t|$ for all $s,t\in\mathbb{R}$,
and $\partial_t^\alpha u$ denotes the Caputo fractional derivative of order $\alpha\in(0,1)$ in time \cite[pp.\,91]{KilbasSrivastavaTrujillo:2006}
\begin{equation}\label{McT}
\partial_t^\alpha u(t):= \frac{1}{\Gamma(1-\alpha)} \int_0^t(t-s)^{-\alpha}\frac{\d}{\d s}u(s)\, \d s ,\quad \mbox{with }\Gamma(z):=\int_0^\infty s^{z-1}e^{-s}\d s.
\end{equation}

Let $S_h\subset H_0^1(\Omega)$ be the continuous piecewise linear finite element space
subject to a quasi-uniform shape regular triangulation of $\Omega$, with a mesh size $h$, and let $\Delta_h:S_h\rightarrow S_h$ denote
the Galerkin finite element approximation of the Dirichlet Laplacian $\Delta$, defined by
$$(\Delta_hw_h,v_h):=-(\nabla w_h,\nabla v_h),\quad \forall\, w_h,v_h\in S_h .$$
Let $0=t_0<t_1<\ldots<t_N=T$ be a uniform partition of the time interval $[0,T]$, with grid
points $t_n=n\tau$ and step size $\tau=T/N$. Upon rewriting the Caputo derivative
$\partial_t^\alpha u$ as a Riemann-Liouville one \cite[pp. 91]{KilbasSrivastavaTrujillo:2006}, we consider a linearized time-stepping scheme:
for the given initial value $u_h^0=R_hu_0$ (Ritz projection of $u_0$), find $u_h^n$, $n=1,2,\ldots, N$, such that
\begin{align}\label{TD-scheme}
\begin{aligned}
&\bar\partial_\tau^\alpha (u_h^n-u_h^0)
-\Delta_h u_h^n=P_hf(u_h^{n-1}) ,
\end{aligned}
\end{align}
where $P_h$ denotes the $L^2$ projection onto the finite element space $S_h$, and $\bar\partial_\tau^\alpha
u_h^n$ denotes either the CQ generated by the backward Euler method or L1 scheme; see
\eqref{eqn:BE} and \eqref{generating-L1} below. These methods are popular for discretizing the fractional derivative in time.

After proving the fractional discrete Gr\"onwall's inequality in Section \ref{sec:Gronwall} and
the regularity estimate in Section \ref{sec:regularity}, we present an error analysis
for the fully discrete scheme \eqref{TD-scheme} in Section \ref{sec:error}. By introducing an intermediate spatially semidiscrete Galerkin problem
\begin{align}\label{nonlinear-FEM}
\partial_t^\alpha u_h(t)-\Delta_h u_h(t)=P_hf(u_h(t)) \quad \forall t\in (0,T],
\end{align}
%discretized in space by the Galerkin finite element method (FEM),
we split the error
into two parts: $u(t_n)-u_h^n=(u(t_n)-u_h(t_n))+(u_h(t_n)-u_h^n)$, and derive the
following error estimates for each component in Theorems \ref{THM:Error-1} and \ref{THM:Error-2}:
\begin{align*}
 \max_{0\le t\le T}\|u(t)-u_h(t)\|_{L^2(\Omega)}\le c\ell_h^2 h^2
 \quad\mbox{and}\quad
 \max_{1\le n\le N}  \|u_h(t_n)-u_h^n\|_{L^2(\Omega)}\le c\tau^\alpha ,
\end{align*}
where $\ell_h=\log(2+1/h)$. %based on the regularity results of $u$ and $u_h$ in Theorem \ref{THM:Reg}.
These estimates are sharp with respect to the regularity of the solution in Theorem \ref{THM:Reg} (up to a logarithmic factor $\ell_h$), and are confirmed by the numerical experiments in Section \ref{sec:numerics}. Besides, we show how to simplify the analysis of nonlinear problems by applying the fractional-type discrete maximal $\ell^p$-regularity established in \cite{JLZ}, an extension of the discrete maximal $\ell^p$-regularity of standard parabolic equations \cite{Kemmochi:2015,KovacsLiLubich2016,LeykekhmanVexler:2017}, which has been applied to numerical analysis of nonlinear parabolic equations in the literature \cite{AkrivisLi2017,AkrivisLiLubich2016,KLC2016}.

Last we mention the interesting works \cite{CuestaLubichPalencia:2006,MustaphaMustapha:2010} on
integro-differential equations, where a Riemann-Liouville fractional integral operator appears in front of the Laplacian.
These models are closely related to \eqref{nonlinear-PDE}, but have different smoothing properties. Cuesta et
al \cite{CuestaLubichPalencia:2006} proposed the CQ generated by the second-order BDF for a semilinear
problem, and proved an $O(\tau^2)$ error bound of the temporal error. In \cite{MustaphaMustapha:2010},
a Crank-Nicolson type method for a semilinear problem with variable time step size was studied. In these
works, a variant of the discrete Gr\"onwall's inequality due to Chen et al
\cite{ChenThomeeWalbin:1992} plays a crucial role,
which differs substantially from the discrete Gr\"onwall's inequality we shall establish below.

Throughout this paper, the notation
$c$ denotes a generic constant, which may vary at different occurrences, but it
is always independent of the mesh size $h$ and time step size $\tau$.

\section{Discrete Gr\"onwall's inequality for time-fractional diffusion}\label{sec:Gronwall}

In this section, we establish
a fractional version of Gr\"onwall's inequality and its discrete
analogue for time stepping schemes. These inequalities are crucial in analyzing numerical
schemes for nonlinear subdiffusion equations, and are of independent interest.

\subsection{Continuous Gr\"onwall's inequality}
We begin with the continuous Gr\"onwall's inequality for fractional differential equations
in a general Banach space setting.

\begin{theorem}[Fractional Gr\"onwall's inequality]\label{Frac-Gronwall}
Let $X$ be any given Banach space.
For $\alpha\in(0,1)$ and $p\in (1/\alpha,\infty)$,
if a function $u\in C([0,T];X)$ satisfies $\partial_t^\alpha u\in L^p(0,T;X)$, $u(0)=0$ and
\begin{align}
\|\partial_t^\alpha u\|_{L^p(0,s;X)}
\le \kappa \|u\|_{L^p(0,s;X)}+\sigma , \quad \forall\, s\in(0,T],
\label{Gronwall-continuous-cond}
\end{align}
for some positive constants $ \kappa$ and $\sigma$, then
\begin{align}\label{Gronwall-continuous-concl}
\|u\|_{C([0,T];X)}+\|\partial_t^\alpha u\|_{L^p(0,T;X)}\le c\sigma,
\end{align}
where the constant $c$ is independent of $\sigma$, $u$ and $X$,
but may depend on $\alpha$, $p$, $\kappa$ and $T$.
\end{theorem}
\begin{proof}
Due to the zero initial condition $u(0)=0$, the Riemann--Liouville and Caputo fractional
derivatives coincide. Hence, the function $u(t)$ can be expressed in terms of $ \partial_t^\alpha u$
(cf. \cite[pp. 96, Lemma 2.22]{KilbasSrivastavaTrujillo:2006} and
\cite[pp. 74, Lemma 2.5]{KilbasSrivastavaTrujillo:2006}):
$u(t)=\frac{1}{\Gamma(\alpha)} \int_0^t (t-\xi)^{\alpha-1} \partial_\xi^\alpha u(\xi) \,\d \xi.$
Since $p>1/\alpha$, H\"older's inequality implies
\begin{equation}\label{eqn:Hardy}
\begin{split}
   \|u(t)\|_{X}
   &\le c \bigg(\int_0^t (t-\xi) ^{\frac{(\al-1)p}{p-1}}\,\d \xi \bigg)^{\frac{p-1}{p}} \| \partial_\xi^\alpha  u\|_{L^p(0,t;X) }   \le c \| \partial_\xi^\alpha u \|_{L^p(0,t;X) }.
\end{split}
\end{equation}
Upon taking the supremum with respect to $t\in(0,s)$ for any $s\in(0,T]$ in \eqref{eqn:Hardy}, we obtain
\begin{equation*}
\begin{aligned}
  \|u\|_{L^\infty(0,s;X)}
  &\le c\|\partial_\xi^\alpha u\|_{L^p(0,s;X)}
  \le c \kappa \|u\|_{L^p(0,s;X)}+c\sigma\\
  &\le \epsilon \kappa \|u\|_{L^\infty(0,s;X)}+c_\epsilon \kappa \|u\|_{L^1(0,s;X)}+c\sigma ,
  \quad\forall\, s\in[0,T] ,
\end{aligned}
\end{equation*}
where $\epsilon>0$ can be arbitrary. By choosing $\epsilon=\frac{1}{2\kappa}$, the $L^\infty$-norm
on the right-hand side can be eliminated by the left-hand side, and the last inequality reduces to
\begin{equation*}
  \|u\|_{L^\infty(0,s;X)}
   \le c_\kappa \|u\|_{L^1(0,s;X)}+c\sigma,
\quad\forall\, s\in[0,T].
\end{equation*}
That is, we have
 $ \|u(s)\|_X \le c_\kappa \int_0^s\|u(\xi)\|_X \d \xi+c\sigma$
for $s\in(0,T].$ Now the standard Gr\"onwall's inequality yields
\begin{equation*}
\max_{s\in[0,T]}\|u(s)\|_X \le e^{c_\kappa T}c\sigma .
\end{equation*}
Substituting it into \eqref{Gronwall-continuous-cond} yields \eqref{Gronwall-continuous-concl}.
The proof of Theorem \ref{Frac-Gronwall} is complete.
\end{proof}

\subsection{Discrete Gr\"onwall's inequality}

In this part, we establish the discrete analogue of the Gr\"onwall's inequality in Theorem
\ref{Frac-Gronwall} for time stepping schemes that
approximate the fractional derivative $\partial^\alpha_tv(t_n)$ by a discrete
convolution:
\begin{align}\label{conv-approx}
\bar\partial^\alpha_\tau v^n:=\frac{1}{\tau^\alpha}\sum_{j=0}^{n}K_{n-j} v^j ,
\qquad n=0,1,2,\dots
\end{align}
where $v^n$ is an approximation of $v(t_n)$, and $K_j$, $j=0,1,2,\ldots$, are the weights
independent of the time step size $\tau$. Throughout, we denote by $K(\zeta)$ the generating
function of the discrete fractional derivative $\bar\partial_\tau^\alpha$, defined by
\begin{equation}
K(\zeta):=\frac{1}{\tau^\alpha}\sum_{j=0}^\infty K_j \zeta^j ,
\end{equation}
which is an analytic function in the (open) unit disk ${\mathbb D}:=\{z\in{\mathbb C}:
|z|<1\}$, continuously differentiable up to the boundary $\partial{\mathbb D}\backslash \{\pm 1\}$, except for the two points $\pm 1$.
Then we have
\begin{align}\label{generating-frac}
&K(\zeta)\sum_{n=0}^\infty v^{n} \zeta^n
=\sum_{n=0}^\infty (\bar\partial^\alpha_\tau v^n)\zeta^n .
\end{align}

\begin{example}
The CQ generated by the $k^{\rm th}$-order BDF
\cite{Lubich:1986,CuestaLubichPalencia:2006} is given by \eqref{conv-approx},
where the coefficients $K_j$, $j=0,1,\dots$, are determined by the power series expansion
\begin{equation}\label{BDF-coefficient}
\bigg(\sum_{j=1}^k \frac {1}{j} (1-\zeta)^j \bigg)^{\alpha}= \sum_{j=0}^\infty K_j \zeta^j .
\end{equation}
The special case $k=1$, i.e., the backward Euler CQ, is very popular and commonly known
as Gr\"{u}nwald--Letnikov approximation, and the coefficients $K_j$,
$j=0,1,2,\ldots$, are  given by
\begin{equation}\label{eqn:BE}
(1-\zeta)^\alpha = \sum_{j=0}^\infty K_j\zeta^j .
\end{equation}
\end{example}

\begin{example}
The popular L1 scheme \cite{LinXu:2007} is also of the form
\eqref{conv-approx} with \cite[pp. 8]{JLZ}
\begin{equation}\label{generating-L1}
\frac{(1-\zeta)^2}{\zeta\Gamma(2-\alpha)}\mathrm{Li}_{\alpha-1}(\zeta)
=\sum_{j=0}^\infty K_j \zeta^j ,
\end{equation}
where $\mathrm{Li}_p(z)=\sum_{j=1}^\infty z^j/j^p$ is the polylogarithmic function, which
is well defined for $|z| < 1$ and can be analytically continued to the split
complex plane $\mathbb{C}\setminus [1,\infty)$ \cite{Flajolet:1999}.
\end{example}

Now we turn to the discrete Gr\"onwall's inequality.
For $1\le p\leq \infty$, we denote by $\ell^p(X)$ the space of sequences $v^n\in X$, $n=0,1,\dots$,
such that
$\|(v^n)_{n=0}^\infty\|_{\ell^p(X)}<\infty$, where
$$
\|(v^n)_{n=0}^\infty\|_{\ell^p(X)}:=
\left\{
\begin{aligned}
&\bigg(\sum_{n=0}^\infty\tau\|v^n\|_{X}^p\bigg)^{\frac{1}{p}}  &&\mbox{if}\,\,\, 1\le p<\infty,\\
&\sup_{n\ge 0}\|v^n\|_{X} &&\mbox{if}\,\,\, p=\infty .
\end{aligned}\right.
$$
For a finite sequence $v^n\in X$, $n=0,1,\dots,m$, we denote
$\|(v^n)_{n=0}^m\|_{\ell^p(X)}:=\|(v^n)_{n=0}^\infty\|_{\ell^p(X)}$,
by setting $v^n=0$ for $n>m$. The following theorem is a discrete
analogue of Theorem \ref{Frac-Gronwall} for the backward Euler CQ. It is foundational to
the proof of the discrete Gr\"onwall's inequalities for other time-stepping schemes.

\begin{theorem}[Discrete fractional Gr\"onwall's inequality: backward Euler]\label{Frac-Gronwall-Euler}
Let $X$ be any given Banach space, and let $\bar\partial_\tau^\alpha\,$ denote the backward Euler CQ
given by \eqref{conv-approx} and \eqref{eqn:BE}.
If $\alpha\in(0,1)$ and $p\in (1/\alpha,\infty)$, and a sequence
$v^n\in X$, $n=0,1,2,\dots$, with $v^0=0$, satisfies
\begin{align}\label{Gronwall-Euler-cond}
\|(\bar\partial_\tau^\alpha v^n)_{n=1}^m\|_{\ell^p(X)}
\le \kappa \|(v^n)_{n=1}^m\|_{\ell^p(X)}+\sigma ,
\quad\forall\, 0\le m\le N ,
\end{align}
for some positive constants $ \kappa$ and $\sigma$, then
there exists a $\tau_0>0$ such that for any $\tau<\tau_0$ there holds
\begin{align}\label{Gronwall-Euler-concl}
\|(v^n)_{n=1}^N\|_{\ell^\infty(X)}+\|(\bar\partial_\tau^\alpha v^n)_{n=1}^N\|_{\ell^p(X)}
\le c  \sigma ,
\end{align}
where the constants $c$ and $\tau_0$ are independent of $\sigma$, $\tau$, $N$, $X$ and $v^n$,
but may depend on $\alpha$, $p$, $\kappa$ and $T$.
\end{theorem}

To prove Theorem \ref{Frac-Gronwall-Euler}, we need a technical lemma, which gives a discrete
analogue of the Hardy type inequality \eqref{eqn:Hardy}.
\begin{lemma}[Discrete Hardy type inequality]\label{Tech-Lemma-Gronwall}
Let $\alpha\in(0,1)$, and $X$ be any given Banach space. If $v^n\in X$ and $w^n\in X$, $n=0,1,2,\dots,$ satisfy
\begin{align} \label{vn-interms-wn}
&\bigg(\frac{1-\zeta}{\tau}\bigg)^\alpha \sum_{n=0}^\infty v^{n} \zeta^n
=\sum_{n=0}^\infty w^n\zeta^n ,
\end{align}
in the sense that both sides are analytic in ${\mathbb D}$, then for $p\in (1/\alpha,\infty)$, there holds
\begin{align}\label{en-intermsof-En}
\|(v^{n})_{n=0}^m\|_{\ell^\infty(X)}
\le c \|(w^n)_{n=0}^m\|_{\ell^{p}(X)} ,
\quad\ 0\le m\le N,
\end{align}
where the constant $c$ is independent of $\tau$, $m$, $N$ and $X$, but may depend on
$\alpha$, $p$ and $T$.
\end{lemma}
\begin{proof}
We define $\phi^n$, $n=0,1,\dots$, to be the coefficients of the power series expansion
$$
(1-\zeta)^{-\alpha}
=\sum_{n=0}^\infty \phi^n\zeta^n .
$$
Then direct calculations yield $\phi^0=1$ and $\phi^n= \prod_{j=1}^n\left(1+\frac{\alpha-1}{j}\right)$
for $n\ge 1$. By the trivial inequality $\ln (1+x)\leq x$ for $x> -1$, we have
\begin{equation*}
\begin{aligned}
\ln\phi^n=  \sum_{j=1}^n \ln\bigg(1+\frac{\alpha-1}{j}\bigg)& \leq (\alpha-1)\sum_{j=1}^nj^{-1}
\le  (\alpha-1)\ln (n+1) .
  \end{aligned}
\end{equation*}
That is, $\phi^n \le (n+1)^{\alpha-1}$ for $n\ge 0 .$
It follows from \eqref{vn-interms-wn} that
\begin{align*}
\sum_{n=0}^\infty v^{n} \zeta^n
=\bigg(\frac{\tau}{1-\zeta}\bigg)^\alpha \sum_{n=0}^\infty w^n \zeta^n
=\tau^\alpha\Big( \sum_{n=0}^\infty \phi^{n} \zeta^n \Big)\Big(\sum_{n=0}^\infty w^n\zeta^n\Big).
\end{align*}
With $p^\prime=\frac{p}{p-1}$, the last identity yields
\begin{align}\label{expr-en-0}
\begin{aligned}
\|v^n\|_X
=\bigg\|\tau^\alpha\sum_{j=0}^n\phi^{n-j} w^j\bigg\|_X
&\le \tau^\alpha\bigg(\sum_{j=0}^n|\phi^{n-j}|^{p'}\bigg)^{\frac{1}{p'}}\bigg(\sum_{j=0}^n\|w^j\|_X^p\bigg)^{\frac{1}{p}} \\
%&=\tau^{\alpha-1/p}  \bigg(\sum_{j=0}^n|\phi^j|^{p'}\bigg)^{\frac{1}{p'}}\bigg(\tau\sum_{j=0}^n\|w^j\|_Y^p\bigg)^{\frac{1}{p}} \\
&\le \tau^{\alpha-1/p} \bigg(\sum_{j=0}^n\frac{1}{(j+1)^{p'(1-\alpha)}}\bigg)^{\frac{1}{p'}}\|(w^j)_{j=0}^n\|_{\ell^p(X)}.
\end{aligned}
\end{align}
If $p>1/\alpha$, then $0<p'(1-\alpha)<1$ and so
$$
\sum_{j=0}^n\frac{1}{(j+1)^{p'(1-\alpha)}}\le
\int_0^{n+1} \frac{\d s}{s^{p'(1-\alpha)}}=
\frac{(n+1)^{1-p'(1-\alpha)}}{1-p'(1-\alpha)}  .
$$
Hence, \eqref{expr-en-0} reduces to
\begin{align*}%\label{expr-en}
%\begin{aligned}
\|v^n\|_X &\le \tau^{\alpha-1/p} \frac{(n+1)^{\alpha-1/p}}{(1-p'(1-\alpha))^{1/p'}}\|(w^j)_{j=0}^n\|_{\ell^p(X)}
\le  \frac{(2T)^{\alpha-1/p}}{(1-p'(1-\alpha))^{1/p'}} \|(w^j)_{j=0}^n\|_{\ell^p(X)}   ,
%\end{aligned}
\end{align*}
where we have used the fact $\tau (n+1)\le 2T$ in the last inequality.
Since the last inequality holds for all $n=0,\dots,m,$
it follows that \eqref{en-intermsof-En} holds.
\end{proof}

Now we are ready to prove Theorem \ref{Frac-Gronwall-Euler}.\smallskip

{\it Proof of Theorem \ref{Frac-Gronwall-Euler}.}$\,\,$
For the backward Euler CQ we have $K(\zeta)=\big(\frac{1-\zeta}{\tau}\big)^\alpha$.
Since, $v^0=0$, $\bar\partial_\tau^\alpha v^0=0$, and
the identity \eqref{generating-frac} can be written as
$\big(\frac{1-\zeta}{\tau}\big)^\alpha\sum_{n=0}^\infty v^{n} \zeta^n
=\sum_{n=0}^\infty (\bar\partial^\alpha_\tau v^n)\zeta^n .$
Then Lemma \ref{Tech-Lemma-Gronwall} and
\eqref{Gronwall-Euler-cond} imply
\begin{align*}
\begin{aligned}
\|(v^n)_{n=0}^m\|_{\ell^\infty(X)}
&\le c\|(\bar\partial_\tau^\alpha v^n)_{n=0}^m\|_{\ell^p(X)}
\le c\kappa \|(v^n)_{n=0}^m\|_{\ell^p(X)}+c\sigma \\
&\le \epsilon\kappa\|(v^n)_{n=0}^m\|_{\ell^\infty(X)}
+c_\epsilon \kappa \|(v^n)_{n=0}^m\|_{\ell^1(X)}+c\sigma ,
\quad\forall\, 1\le m\le N .
\end{aligned}
\end{align*}
By choosing $\epsilon\kappa=1/2$ and collecting terms, and using the fact $v^0=0$, we obtain
\begin{align*}
\begin{aligned}
\|(v^n)_{n=1}^m\|_{\ell^\infty(X)}
\le c_\kappa \|(v^n)_{n=1}^m\|_{\ell^1(X)}+c\sigma ,
\quad\forall\, 1\le m\le N .
\end{aligned}
\end{align*}
That is, $\|v^m\|_{X}\le c_\kappa \tau \sum_{n=1}^m\|v^n\|_{X}+c\sigma$ for $1\le m\le N$.
Then the standard discrete Gr\"onwall's inequality gives, for sufficiently small  step size $\tau$,
\begin{align*}
\max_{1\le n\le N} \|v^n\|_{X} \le e^{c_\kappa T}c\sigma .
\end{align*}
Substituting this into \eqref{Gronwall-Euler-cond} yields \eqref{Gronwall-Euler-concl}.
The proof of Theorem \ref{Frac-Gronwall-Euler} is complete.
\endproof

To analyze other time-stepping schemes, we shall need the following lemma of discrete Mikhlin multipliers, which is
a simple consequence of Blunck's multiplier theorem \cite[Theorem 1.3]{Blunck:2001} through the transform
$\zeta=e^{-\mathrm{i}\theta}$. Here, a UMD space $X$ denotes a Banach space such that the Hilbert transform
$Hf(t) : =\int_\mathbb{R} \frac{f(s)}{t-s}\d s$ is bounded on $L^p(\mathbb{R};X)$ for all $1<p<\infty$
\cite{KunstmannWeis:2004}. Examples of UMD spaces include ${\mathbb R}^d$, $d\ge 1$, and $L^q(\Omega)$,
$1<q<\infty$, and their closed subspaces (e.g. the finite element space $S_h$ equipped with the $L^q(\Omega)$ norm).

\begin{lemma}[Discrete Mikhlin multipliers]\label{lem:Blunck}
Let $X$ be a UMD space and let $M: {\mathbb D}\to {\mathbb C}$ be an analytic function,
continuously differentiable up to $\partial{\mathbb D}\backslash\{\pm 1\}$, such that the set
\begin{equation*}
\big\{ M(\zeta):\,\zeta \in \partial{\mathbb D}\backslash\{\pm 1\} \big\} \, \cup\, %\quad\hbox{and}\quad
\big\{(1-\zeta)(1+\zeta) M'(\zeta): \, \zeta \in \partial{\mathbb D}\backslash\{\pm 1\}\big\}
\end{equation*}
is bounded, and denote its bound by $c_R$.
Then for any $1 < p < \infty$ and any
sequence $(f^n)_{n=0}^\infty\in \ell^p(X)$,
the coefficients $u_n\in X$, $n=0,1,\dots$, in the power series expansion
\begin{align*}
M(\zeta) \sum_{n=0}^\infty f^n\zeta^n=\sum_{n=0}^\infty u^n\zeta^n ,
\qquad\forall\,\zeta\in {\mathbb D} ,
%\,\,\,\zeta\in {\mathbb D},
\end{align*}
satisfy
\begin{align*}
\|(u^n)_{n=0}^\infty \|_{\ell^p(X)}
\le c_{p,X}c_R\|(f^n)_{n=0}^\infty\|_{\ell^p(X)} ,
\end{align*}
where the constant $c_{p,X}$ is independent of the operators $M(\zeta)$,
$\zeta \in {\mathbb D}$.
\end{lemma}

Now other time-stepping schemes can be connected to the backward Euler CQ. The next result
gives a general criterion for the discrete fractional Gr\"onwall's inequality.
\begin{theorem}[General criterion for discrete fractional Gr\"onwall's inequality]
\label{Other-Scheme-A2} Let $X$ be a UMD space. If the generating function $K(\zeta)=\frac{1}{\tau^\alpha} \sum_{n=0}^\infty K_n\zeta^n$ satisfies
\begin{equation} \label{eqn:est_K}
| K(\zeta) |\ge \frac{1}{c}\bigg|\frac{1-\zeta}{\tau}\bigg|^\alpha\quad \mbox{and}\quad |(1-\zeta)(1+\zeta)K'(\zeta)|\le c|K(\zeta)|, \quad\forall\, \zeta\in
\partial\mathbb{D}\backslash\{\pm 1\} ,
\end{equation}
then the discrete fractional Gr\"onwall's inequality holds:
if $\alpha\in(0,1)$ and $p\in (1/\alpha,\infty)$,
and a sequence $v^n\in X$, $n=0,1,2,\dots$, with $v^0=0$, satisfies
\begin{align}\label{Gronwall-other-cond}
\|(\bar\partial_\tau^\alpha v^n)_{n=1}^m\|_{\ell^p(X)}
\le \kappa \|(v^n)_{n=1}^m\|_{\ell^p(X)}+\sigma ,
\quad\forall\, 1\le m\le N ,
\end{align}
for some positive constants $ \kappa$ and $\sigma$, then
there exists a $\tau_0>0$ such that for any $\tau<\tau_0$ there holds
\begin{align}\label{Gronwall-other-concl-2}
\|(v^n)_{n=1}^N\|_{\ell^\infty(X)}+\|(\bar\partial_\tau^\alpha v^n)_{n=1}^N\|_{\ell^p(X)}
\le c \sigma .
\end{align}
where the constants $c$ and $\tau_0$ are independent of $\sigma$, $\tau$, $N$ and $v^n$,
but may depend on $\alpha$, $p$, $\kappa$, $X$ and $T$.
\end{theorem}
\begin{proof}
First, we note that $\bar\partial^\alpha_\tau v^n=\tau^{-\alpha}\sum_{j=0}^{n}K_j v^{n-j}$, $n=0,1,2,\dots$,
are the coefficients in the power series expansion
\begin{align}
&K(\zeta)\sum_{n=0}^\infty v^{n} \zeta^n
=\sum_{n=0}^\infty (\bar\partial^\alpha_\tau v^n)\zeta^n  ,
\end{align}
it follows that
\begin{align}\label{eqn:Fn_zeta}
\sum_{n=0}^\infty v^n\zeta^n
=\bigg(\frac{\tau}{1-\zeta}\bigg)^{\alpha} \bigg[\frac{1}{K(\zeta)}\bigg(\frac{1-\zeta}{\tau}\bigg)^\alpha\bigg]
\sum_{n=0}^\infty  (\bar\partial^\alpha_\tau v^n)\zeta^n =\bigg(\frac{\tau}{1-\zeta}\bigg)^{\alpha}
\sum_{n=0}^\infty F^n\zeta^n ,
\end{align}
where $F^n$, $n=0,1,\dots$, are the coefficients in the expansion
\begin{equation*}
\bigg[\frac{1}{K(\zeta)}\bigg(\frac{1-\zeta}{\tau}\bigg)^\alpha\bigg]
\sum_{n=0}^\infty  (\bar\partial^\alpha_\tau v^n)\zeta^n =\sum_{n=0}^\infty F^n\zeta^n  .
\end{equation*}
By applying Lemma \ref{Tech-Lemma-Gronwall} to \eqref{eqn:Fn_zeta}, we obtain
\begin{equation} \label{Est-en-Fn}
\|(v^n)_{n=0}^m\|_{\ell^\infty(X)}
\le c \|(F^n)_{n=0}^m\|_{\ell^{p}(X)} ,\quad \forall\, 1\le m\le N .
\end{equation}
Let $m$ be fixed and define $\widetilde E^n=\bar\partial^\alpha_\tau v^n$ if $n\le m$ and
$\widetilde E^n=0$ if $n>m$. Let $\widetilde F^n$ be the coefficients of the power series
\begin{equation}\label{Eq_en_zeta}
\sum_{n=0}^\infty\widetilde F^n\zeta^n
=\bigg[\frac{1}{K(\zeta)}\bigg(\frac{1-\zeta}{\tau}\bigg)^\alpha\bigg]
\sum_{n=0}^\infty \widetilde E^n\zeta^n ,
\end{equation}
then $\widetilde F^n=F^n$ for $0\le n\le m$.
Now the conditions in \eqref{eqn:est_K} imply
\begin{equation*}
\bigg|\frac{1}{K(\zeta)} \bigg(\frac{1-\zeta}{\tau}\bigg)^\alpha\bigg| \le c \quad\mbox{and}\quad
\bigg|(1-\zeta)(1+\zeta)\frac{\d}{\d \zeta}\bigg[\frac{1}{K(\zeta)} \bigg(\frac{1-\zeta}{\tau}\bigg)^\alpha\bigg]\bigg|\le c, \quad \forall \zeta\in \partial\mathbb{D}\setminus\{\pm1\} .
\end{equation*}
By choosing $M(\zeta)=\frac{1}{K(\zeta)} \big(\frac{1-\zeta}{\tau}\big)^\alpha$ and applying Lemma \ref{lem:Blunck} to equation
\eqref{Eq_en_zeta}, we obtain
\begin{equation*}
\|(\widetilde F^n)_{n=0}^\infty\|_{\ell^{p}(X)}
\le c\|(\widetilde E^n)_{n=0}^\infty\|_{\ell^{p}(X)} ,
\end{equation*}
which further implies
\begin{equation*}
\|(F^n)_{n=0}^m\|_{\ell^{p}(X)}
= \|(\widetilde F^n)_{n=0}^m\|_{\ell^{p}(X)}
\le c \|(\widetilde E^n)_{n=0}^\infty\|_{\ell^{p}(X)}
=c \|(\bar\partial^\alpha_\tau v^n)_{n=0}^m\|_{\ell^{p}(X)} ,
\end{equation*}
where the constant $c$ is independent of $m$. The last inequality and \eqref{Est-en-Fn} yield
\begin{equation*}
\|(v^n)_{n=0}^m\|_{\ell^\infty(X)}
\le c \|(\bar\partial^\alpha_\tau v^n)_{n=0}^m\|_{\ell^{p}(X)} .
\end{equation*}
Substituting \eqref{Gronwall-other-cond} into the last inequality gives
\begin{align}\label{eqn:v-est}
\begin{aligned}
\|(v^n)_{n=1}^m\|_{\ell^\infty(X)}
&\le c\kappa \|(v^n)_{n=1}^m\|_{\ell^p(X)}+c\sigma \\
&\le \epsilon\kappa\|(v^n)_{n=1}^m\|_{\ell^\infty(X)}
+c_\epsilon \kappa \|(v^n)_{n=1}^m\|_{\ell^1(X)}+c\sigma ,
\quad\forall\, 1\le m\le N .
\end{aligned}
\end{align}
where $\epsilon>0$ is arbitrary. By choosing $\epsilon\kappa=1/2$, we obtain
\begin{align*}
\begin{aligned}
\|(v^n)_{n=1}^m\|_{\ell^\infty(X)}
&\le c_\kappa \|(v^n)_{n=1}^m\|_{\ell^1(X)}+c\sigma ,
\quad\forall\, 1\le m\le N .
\end{aligned}
\end{align*}
That is, $\|v^m\|_{X}\le c_\kappa \tau \sum_{n=1}^m\|v^n\|_{X}+c\sigma$ for $1\le m\le N$.
Then the standard discrete Gr\"onwall's inequality gives, for sufficiently small step size $\tau$,
\begin{align*}
\max_{1\le n\le N}\|v^n\|_{X} \le e^{c_\kappa T}c\sigma .
\end{align*}
This together with \eqref{Gronwall-other-cond} and \eqref{eqn:v-est} yields \eqref{Gronwall-other-concl-2}.
The proof of Theorem \ref{Other-Scheme-A2} is complete.
\end{proof}

By Theorem \ref{Other-Scheme-A2}, the discrete fractional
Gr\"onwall's inequality can be proved for the L1 scheme and general BDF CQs.

\begin{theorem}[Discrete Gr\"onwall's inequality for L1 scheme and BDF CQ]\label{L1-CQ-gronwall}$\,$\\
Let $X$ be a UMD space. For both L1 scheme and CQ generated by the $k^{\rm th}$-order BDF, with $1\le
k\le 6$, the discrete fractional Gr\"onwall's inequality holds: if $\alpha\in(0,1)$ and $p\in
(1/\alpha,\infty)$, and a sequence $v^n\in X$, $n=0,1,2,\dots$, with $v^0=0$, satisfies
\begin{align*}
\|(\bar\partial_\tau^\alpha v^n)_{n=1}^m\|_{\ell^p(X)}
\le \kappa \|(v^n)_{n=1}^m\|_{\ell^p(X)}+\sigma ,\quad\forall\, 1\le m\le N ,
\end{align*}
for some positive constants $ \kappa$ and $\sigma$,
then there exists a $\tau_0>0$ such that for any $\tau<\tau_0$ there holds
\begin{align*}
\|(v^n)_{n=1}^N\|_{\ell^\infty(X)}+\|(\bar\partial_\tau^\alpha v^n)_{n=1}^N\|_{\ell^p(X)}\le c\sigma,
\end{align*}
where the constants $c$ and $\tau_0$ are independent of $\sigma$, $\tau$, $N$ and $v^n$,
but may depend on $\alpha$, $p$, $\kappa$, $X$ and $T$.
\end{theorem}
\begin{proof}
By Theorem \ref{Other-Scheme-A2}, it suffices to show that the generating functions $K(\zeta)$ of the L1
scheme and CQ satisfy \eqref{eqn:est_K}. We discuss them separately.
First, for the L1 scheme, $K(\zeta)=\frac{1}{\Gamma(2-\alpha)\tau^\alpha}\frac{(1-\zeta)^2}{\zeta}
\mathrm{Li}_{\alpha-1}(\zeta)$ converges for $\zeta\in\partial{\mathbb D}\backslash\{1\}$ and has
the following asymptotic expansion (cf. \cite[Theorem 1]{Flajolet:1999}, or \cite[equation (4.6)]{JLZ})
\begin{equation*}%\label{L1-asymp1}
\tau^{\alpha} K(\zeta) = (1-\zeta)^{\alpha} + o((1-\zeta)^{\alpha})  ,
\quad \mbox{as}\,\,\, \zeta\rightarrow 1.
\end{equation*}
If $\zeta\in \partial{\mathbb D}\backslash\{1\}$ is sufficiently close to $1$, then
$$\tau^\alpha |K(\zeta)|\ge \tfrac{1}{2}|1-\zeta|^\alpha.$$
Meanwhile, we recall the following series expansion (cf. \cite[equation (4.5)]{JLZ})
\begin{equation*}
\begin{aligned}
\frac{\mathrm{Li}_{\alpha-1}(e^{-{\rm i}\theta})}{\Gamma(2-\alpha)}
        & =(2\pi)^{\alpha-2}\bigg(
       \cos\big(\frac{(2-\alpha)\pi}{2}\big)(A_{\theta}+B_{\theta})
-   {\mathrm i}\sin\big(\frac{(2-\alpha)\pi}{2}\big) (A_{\theta}-B_{\theta})  \bigg),
\end{aligned}
\end{equation*}
where $A_{\theta}=\sum_{k=0}^\infty\left(k+\frac{\theta}{2\pi}\right)^{\alpha-2}$
and $ B_{\theta}=\sum_{k=0}^\infty\left(k+1-\frac{\theta}{2\pi}\right)^{\alpha-2}.$
Thus, if $\zeta=e^{-{\rm i}\theta}$ is away from $1$, then $\theta$ is away from $0$ and $2\pi$, and thus
$A_{\theta}+B_{\theta}\ge c$. This shows $|\mathrm{Li}_{\alpha-1}(e^{-{\rm i}\theta})|>c$.
Since $|1-\zeta|^2\geq c|1-\zeta|^\alpha$ when $\zeta=e^{-{\rm i}\theta}$ is away from $1$, it follows that
\begin{equation*}
  \tau^\alpha |K(\zeta)|=\frac{|\mathrm{Li}_{\alpha-1}(\zeta)|}{\Gamma(2-\alpha)}|1-\zeta|^2\ge c|1-\zeta|^2\ge c|1-\zeta|^\alpha.
\end{equation*}
Overall, the first inequality of \eqref{eqn:est_K} holds for the generating function $K(\zeta)$ of the L1 scheme.
The second inequality of \eqref{eqn:est_K} has been proved in \cite[Lemma 4.3]{JLZ}. This shows
the assertion for the L1 scheme.

Next we turn to the CQ. For the CQ generated by the $k^{\rm th}$-order BDF, the generating function $K(\zeta)$ satisfies
\begin{equation*}
\bigg(\frac{\tau}{1-\zeta}\bigg)^\alpha  K(\zeta)
=\bigg(\sum_{j=1}^k \frac {1}{j} (1-\zeta)^{j-1} \bigg)^{\alpha} .
\end{equation*}
Since the function $\sum_{j=1}^k \frac {1}{j} (1-\zeta)^{j-1}$ has no root on the unit circle $\partial\mathbb{D}$
for $1\le k\le 6$ (see \cite[Proof of Lemma 2]{CreedonMiller:1975}
or \cite[pp. 246-247]{HairerNorsettWanner:2010}), it follows that
\begin{equation*}
\bigg|\bigg(\frac{\tau}{1-\zeta}\bigg)^\alpha K(\zeta) \bigg|=\bigg|\sum_{j=1}^k \frac {1}{j} (1-\zeta)^{j-1}\bigg|^\alpha
\ge c .
\end{equation*}
This proves the first inequality of \eqref{eqn:est_K}. Note that
\begin{equation*}
\begin{aligned}
(1+\zeta)(1-\zeta)K'(\zeta)
&=
-(1+\zeta)(1-\zeta) \frac{\alpha}{\tau^{\alpha}}\bigg(\sum_{j=1}^k \frac {1}{j} (1-\zeta)^{j} \bigg)^{\alpha-1}
\sum_{j=1}^{k} (1-\zeta)^{j-1} \\
&=
- \frac{\alpha}{\tau^{\alpha}}(1+\zeta)(1-\zeta)^\alpha \bigg(\sum_{j=1}^k \frac {1}{j} (1-\zeta)^{j-1} \bigg)^{\alpha-1}
\sum_{j=0}^{k-1} (1-\zeta)^{j},
\end{aligned}
\end{equation*}
and so for any $\zeta\in\partial{\mathbb D}\backslash\{\pm 1\}$, there holds
\begin{equation*}
\begin{aligned}
\bigg|\frac{(1+\zeta)(1-\zeta)K'(\zeta)}{K(\zeta)}\bigg|=
\bigg|\frac{\alpha(1+\zeta)\big(\sum_{j=1}^k \frac {1}{j} (1-\zeta)^{j-1} \big)^{\alpha-1}
\sum_{j=0}^{k-1} (1-\zeta)^{j}}{\big(\sum_{j=1}^k \frac {1}{j} (1-\zeta)^{j-1} \big)^{\alpha}}\bigg|
\le c.
\end{aligned}
\end{equation*}
where the last inequality holds, since the denominator
$\big(\sum_{j=1}^k \frac {1}{j} (1-\zeta)^{j-1} \big)^{\alpha}$ has no root on $\partial{\mathbb D}$.
This shows the second part of \eqref{eqn:est_K}, completing the  proof of the theorem.
\end{proof}

\begin{remark}
{\upshape
In Theorems  \ref{Other-Scheme-A2} and \ref{L1-CQ-gronwall}, if we assume
\begin{align*}
\|(\bar\partial_\tau^\alpha v^n)_{n=1}^m\|_{\ell^p(X)}
\le \kappa \|(v^n)_{n=1}^{m-1}\|_{\ell^p(X)}+\sigma ,
\quad\forall\, 1\le m\le N ,
\end{align*}
i.e., the index on the right-hand side is slightly changed, then we have
\begin{align*}
\|(v^n)_{n=1}^N\|_{\ell^\infty(X)}+\|(\bar\partial_\tau^\alpha v^n)_{n=1}^N\|_{\ell^p(X)}
\le c  \sigma ,
\end{align*}
without any restriction on the step size $\tau$.}
\end{remark}

\section{Regularity of the solution}\label{sec:regularity}
Now we discuss the existence, uniqueness and regularity for the solutions
to \eqref{nonlinear-PDE} and  \eqref{nonlinear-FEM}.
These results are needed in the numerical analysis in Section \ref{sec:error}.
The main result of this section is the following theorem.

\begin{theorem}\label{THM:Reg}
Let $u_0\in H_0^1(\Omega)\cap H^2(\Omega)$, and let $f:{\mathbb R}\rightarrow {\mathbb R}$ be Lipschitz continuous.
Then problem \eqref{nonlinear-PDE} has a unique solution $u$ such that
\begin{align}
&u\in C^\alpha([0,T];L^2(\Omega)) \cap
C([0,T];H^1_0(\Omega)\cap H^2(\Omega)) ,\quad
\partial_t^\alpha u\in C([0,T];L^2(\Omega)) , \label{reg-PDE}\\
&\partial_tu(t)\in L^2(\Omega)\quad
\mbox{and}\quad
 \|\partial_tu(t)\|_{L^2(\Omega)}\le ct^{\alpha-1}
\quad \mbox{for}\,\,\, t\in(0,T] .
 \label{reg-PDE2}
\end{align}
Similarly, problem \eqref{nonlinear-FEM} has a unique solution $u_h$ such that
\begin{align}
&\|u_h\|_{C^\alpha([0,T];L^2(\Omega))}
+\|\Delta_hu_h\|_{C([0,T];L^2(\Omega))} +
\|\partial_t^\alpha u_h\|_{C([0,T];L^2(\Omega))}
\le c, \label{reg-FEM}\\
& \|\partial_tu_h(t)\|_{L^2(\Omega)}\le ct^{\alpha-1}
\quad \mbox{for}\,\,\, t\in(0,T] .
 \label{reg-FEM2}
\end{align}
The constant $c$ above is independent of the mesh size $h$,
but may depend on $T$.
\end{theorem}

\begin{remark}\label{Reg-non-Lipschitz}
{\upshape
For smooth initial data and right-hand side, in the absence of extra compatibility conditions, the
regularity results \eqref{reg-PDE}-\eqref{reg-PDE2} and the $h$-independent estimates
\eqref{reg-FEM}-\eqref{reg-FEM2} are sharp with respect to the H\"older continuity in time.
The regularity \eqref{reg-PDE} was shown in \cite{SakamotoYamamoto:2011} for linear subdiffusion equations
and in \cite{Luchko:2013} for a semilinear problem with Neumann boundary conditions under certain
compatibility conditions. However, we are not aware of any existing results such as \eqref{reg-PDE2} and
\eqref{reg-FEM}-\eqref{reg-FEM2} for semilinear problems without compatibility conditions, which are
important for the numerical analysis in Section \ref{sec:error}.
%For self-containedness, we provide a proof of Theorem \ref{THM:Reg} in this section.
}
\end{remark}

\begin{remark}
{\upshape If $f$ is smooth but not Lipschitz continuous,
and problems \eqref{nonlinear-PDE} and \eqref{nonlinear-FEM} have unique bounded solutions,
respectively, then $f(u)$, $f'(u)$, $f(u_h)$ and $f'(u_h)$ are still bounded.
In this case, the estimates \eqref{reg-PDE}-\eqref{reg-PDE2} and
\eqref{reg-FEM}-\eqref{reg-FEM2} are still valid, which
can be seen from the proof of Theorem \ref{THM:Reg}.}
\end{remark}

We begin with some preliminary results. Let $L^2_h(\Omega)$ be the vector space $S_h$
equipped with the norm of $L^2(\Omega)$ and let $H^2_h(\Omega)$ be the vector space $S_h$
equipped with the norm
\begin{align*}
\|v_h\|_{H^2_h(\Omega)}:=
\|v_h\|_{L^2(\Omega)}+\|\Delta_hv_h\|_{L^2(\Omega)},
\quad \forall\, v_h\in S_h .
\end{align*}
To analyze $u(t)$ and $u_h(t)$ in a unified way, we consider the following abstract problem:
\begin{align}\label{nonlinear-abstr}
\left\{
\begin{aligned}
&\partial_t^\alpha u(t)- Au(t)=P f(u(t)) &&\mbox{for}\,\,\,t\in (0,T],\\
&u(0)=u_0 ,
\end{aligned}
\right.
\end{align}
where the notation $(X,D,A,u,P,u_0)$ denotes either
$(L^2(\Omega),H^1_0(\Omega)\cap H^2(\Omega),\Delta,u,I,u_0)$ or
$(L^2_h(\Omega),H^2_h(\Omega),\Delta_h,u_h,P_h,R_hu_0)$,
with $I$ denoting the identity operator.
In a bounded convex polygonal domain $\Omega$, the norm of $D$ is equivalent to the graph norm, i.e.,
\begin{align}\label{norm-D}
\|v\|_D \sim \|v\|_X+\|Av\|_X,\quad\forall\, v\in D.
\end{align}
Let $\|\cdot\|_{X\to X}$ be the operator norm on the space $X$. Then the operator $A$ satisfies the following resolvent estimate
\cite[Example 3.7.5 and Theorem 3.7.11]{ABHN}:% {\color{red}[2] contains also $A_h$?}
\begin{equation*}%\label{Deltah-resolvent}
  \| (z -A)^{-1} \|_{X\rightarrow X}\le c_\phi |z|^{-1},  \quad \forall z \in \Sigma_{\phi},
  \,\,\,\forall\,\phi\in(0,\pi) ,
\end{equation*}
where for $\phi\in(0,\pi)$, $\Sigma_{\phi}:=\{z\in{\mathbb C}\backslash\{0\}: |{\rm arg}(z)|<\phi\}$.
This further implies
\begin{equation}\label{eqn:resol}
\begin{aligned}
 & \| (z^{\alpha}-A)^{-1} \|_{X\rightarrow X}
\le c_{\phi,\alpha} |z|^{-\alpha},
 &&\forall z \in \Sigma_{\phi} ,
 \,\,\,\forall\,\phi\in(0,\pi) ,\\
& \| A(z^{\alpha}-A)^{-1} \|_{X\rightarrow X} \le c_{\phi,\alpha} ,
&&\forall z \in \Sigma_{\phi} ,
 \,\,\,\forall\,\phi\in(0,\pi) .
 \end{aligned}
\end{equation}

Let $g(t)=P f(u(t))$, and $w:=u-u_0$. Then $w$ satisfies the following equation
\begin{equation}\label{PDE-wt}
    \partial_t^\alpha w(t)   - A w(t)  = A u_0 + g(t)  ,
\end{equation}
with $w(0)=0$. By means of Laplace transform, denoted by $~\widehat{}~$, we obtain
\begin{equation*}
  z^\alpha \widehat{w}(z) - A\widehat w(z) = z^{-1}Au_0+\widehat {g}(z),
\end{equation*}
which together with \eqref{eqn:resol} implies $\widehat{w}(z)=(z^\alpha-A)^{-1}(z^{-1}Au_0+\widehat{g}(z))$. By
inverse Laplace transform and convolution rule, the solution
$w(t)$ to \eqref{PDE-wt} is given by
\begin{align}\label{Int-Eq-w}
w(t)=F(t)A  u_0+\int_0^t E(t-s)g(s)\d s ,
\end{align}
where the operators $F(t):X \to X $ and $E(t):X \to X $ are defined by %[{\color{red}give a reference?}]
\begin{equation}\label{eqn:EF}
F(t):=\frac{1}{2\pi {\rm i}}\int_{\Gamma_{\theta,\delta }}e^{zt} z^{-1} (z^\alpha-A )^{-1}\, \d z
\quad\mbox{and}\quad E(t):=\frac{1}{2\pi {\rm i}}\int_{\Gamma_{\theta,\delta}}e^{zt}  (z^\alpha-A )^{-1}\, \d z ,
\end{equation}
respectively. Clearly, we have $E(t)=F^\prime(t)$. The contour $\Gamma_{\theta,\delta}$ is defined by
\begin{equation}\label{contour-Gamma}
  \Gamma_{\theta,\delta}=\left\{z\in \mathbb{C}: |z|=\delta , |\arg z|\le \theta\right\}\cup
  \{z\in \mathbb{C}: z=\rho e^{\pm {\rm i}\theta}, \rho\ge \delta \},
\end{equation}
oriented with an increasing imaginary part, where $\theta\in(\pi/2,\pi)$ is fixed. In view of \eqref{Int-Eq-w},
$u$ is the solution of problem \eqref{nonlinear-abstr} if and only if it is the solution of
\begin{align}\label{re-form-nonlinear}
u(t)-u_0=F(t)A  u_0+\int_0^t E(t-s)  P f(u(s)) \d s .
\end{align}

The next lemma summarizes the mapping properties of the operators $F$ and $E$.
These are partially known \cite[Section 2]{SakamotoYamamoto:2011} and \cite{McLean:2010}. We only
sketch the proof for completeness.
\begin{lemma}\label{lem:smoothing}
For the operators $F$ and $E$, the
following properties hold.
\begin{itemize}
\item[$\rm(i)$] $t^{-\alpha}\|F(t)\|_{X \rightarrow X }
+t^{1-\alpha}\|F'(t)\|_{X \rightarrow X }
+  \|A  F(t)\|_{X \rightarrow X }\le c
,\quad\forall\,t\in(0,T]$ ,
\item[$\rm(ii)$] $F(t):X \rightarrow D$ is continuous with respect to $t\in[0,T]$, and $AF(0)=0$.
\item[$\rm(iii)$] $t^{1-\alpha}\|E(t)\|_{X \rightarrow X }
+t^{2-\alpha}\|E'(t)\|_{X \rightarrow X }
+t\|A  E(t)\|_{X \rightarrow X }
\le c , \quad\forall\,t\in(0,T]$.
\item[$\rm(iv)$]$E(t):X \rightarrow D \,\,\,\mbox{is continuous with respect to $t\in(0,T]$}.$
\end{itemize}
\end{lemma}
\begin{proof}
First, consider (ii) in the case $X=L^2(\Omega)$, $D=H^1_0(\Omega)\cap H^2(\Omega)$ and $A=\Delta$.
By setting $f(u(t))\equiv 0$ and $A=\Delta$ in \eqref{re-form-nonlinear},
\cite[Theorem 2.1]{SakamotoYamamoto:2011} implies that $\Delta F(t)=F(t)\Delta: L^2(\Omega)\rightarrow L^2(\Omega)$ is continuous with respect to $t\in[0,T]$. Thus, $F(t):L^2(\Omega)\rightarrow H^1_0(\Omega)\cap H^2(\Omega)$ is continuous with respect to $t\in[0,T]$.
Then taking $t\rightarrow 0$ in \eqref{re-form-nonlinear} yields
$\Delta F(0)=0$.
This proves (ii) in the case $X=L^2(\Omega)$, $D=H^1_0(\Omega)\cap H^2(\Omega)$ and $A=\Delta$.
The proof for the case $X=L^2_h(\Omega)$, $D=H^2_h(\Omega)$ and $A=\Delta_h$ is similar.
% since all the norms of on the finite dimensional space $L^2_h(\Omega)$ are equivalent.

%In the case $X=L^2_h(\Omega)$, $D=H^2_h(\Omega)$ and $A=\Delta_h$, it is easy to derive from
%\eqref{eqn:EF} that $F(t):X \rightarrow X$ is continuous with respect to $t\in[0,T]$. Since all the norms of the finite dimensional space $L^2_h(\Omega)$ are equivalent,
%it follows that $F(t):X \rightarrow D$ is also continuous for $t\in[0,T]$.
%Then, by setting $f(u(t))\equiv 0$ and taking $t\rightarrow 0$ in \eqref{re-form-nonlinear}, we obtain $\Delta_h F(0)=F(0)\Delta_h =0$.
%This proves (ii) in the case $X=L^2_h(\Omega)$, $D=H^2_h(\Omega)$ and $A=\Delta_h$.

For any integers $k\ge 0$ and $m=0,1$, by choosing $\delta=t^{-1}$ in the contour $\Gamma_{\theta,\delta}$
and using the identity $A(z^\alpha-A)^{-1}=-I+z^\alpha(z^\alpha-A)^{-1}$, the resolvent estimate \eqref{eqn:resol}, and change of variables $z=s\cos\varphi+{\rm i}s\sin\varphi$, we have (with $|\d z|$ being the arc length element of $\Gamma_{\theta,\delta}$)
\begin{align*}
\bigg\|A^{m}\frac{\d^k}{\d t^k}F(t)\bigg\|_{X\rightarrow X}
&=\bigg\|\frac{1}{2\pi {\rm i}}\int_{\Gamma_{\theta,\delta }}e^{zt} z^{k-1} A^{m}(z^\alpha-A )^{-1}\, \d z\bigg\|_{X\rightarrow X} \\
&\le c\int_{\Gamma_{\theta,\delta }}
e^{{\rm Re}(z)t} |z|^{k-1+(m-1)\alpha} \, |\d z| \\
&\le c|\cos\theta| \int_{\delta}^\infty e^{st\cos\theta} s^{k-1+(m-1)\alpha} \d s
+ c\int_{-\theta}^\theta e^{\cos\varphi} \delta^{k+(m-1)\alpha} \d \varphi   \\
&\le ct^{-(m-1)\alpha-k} .
\end{align*}
Since $E(t)=F^\prime(t)$, the last inequality yields (i) and (iii).
The continuity of $F(t):X \rightarrow D$ and $E(t):X \rightarrow D$  for $t\in(0,T]$
follows from the equivalent norm in \eqref{norm-D}, showing (iv).
%\medskip
\end{proof}

Now we are ready to present the proof of Theorem \ref{THM:Reg}.\smallskip

{\it Proof of Theorem \ref{THM:Reg}.}$\,\,$ The proof is divided into four steps.

\noindent{\it Step 1: Existence and uniqueness.}
We denote by $C([0,T];X)_\lambda$ the function space $C([0,T];X)$ equipped with the following weighted norm:
$$
\|v\|_{\lambda}:=
\max_{0\le t\le T}\|e^{-\lambda t}v(t)\|_X ,\quad\forall\, v\in C([0,T];X),
$$
which is equivalent to the standard norm of $C([0,T];X)$
for any fixed parameter $\lambda>0$.
Then we define a nonlinear map
$M:C([0,T];X )_\lambda\rightarrow C([0,T];X )_\lambda$ by
\begin{equation*}%\label{Map-M}
Mv(t)=u_0+F(t)A  u_0+\int_0^t E(t-s) P f(v(s)) \d s .
\end{equation*}
For any $\lambda>0$,
$u\in C([0,T];X)$ is a solution of \eqref{re-form-nonlinear} if and only if $u$ is a fixed point of
the map $M:C([0,T];X )_\lambda\rightarrow C([0,T];X )_\lambda$. It remains to prove that for some $\lambda>0$,
the map $M:C([0,T];X )_\lambda\rightarrow C([0,T];X )_\lambda$ has a unique fixed point.
In fact, the definition of $M$ and Lemma \ref{lem:smoothing}(iii) immediately yield
\begin{equation}\label{map-m-contr-1}
\begin{aligned}
&\quad \|e^{-\lambda t}(Mv_1(t)-Mv_2(t))\|_{X } \\
&=\bigg\|e^{-\lambda t}\int_0^t E(t-s)(Pf(v_1(s))-Pf(v_2(s)))\d s\bigg\|_{X }\\
&\le ce^{-\lambda t}\int_0^t(t-s)^{\alpha-1}\|v_1(s)-v_2(s)\|_{X}\d s\\
&\le c\int_0^t(t-s)^{\alpha-1}e^{-\lambda(t-s)}
\max_{s\in[0,T]}\|e^{-\lambda s}(v_1(s)-v_2(s))\|_{X}\d s\\
&= c\lambda^{-\alpha}  \bigg(\int_0^1(1-\theta)^{\alpha-1}
(\lambda t)^\alpha e^{-\lambda t(1-\theta)}\d \theta\bigg) \|v_1-v_2\|_{\lambda} \quad\mbox{(change of variable $s=t\theta$)} \\
&\le c\sup_{\begin{subarray}{c}
\lambda>0,T\geq t>0\\
\theta\in[0,1]
\end{subarray}}
\Big([\lambda t(1-\theta)]^{\alpha/2}e^{-\lambda t(1-\theta)}\Big)
(t/\lambda)^{\alpha/2}  \bigg(\int_0^1(1-\theta)^{\alpha/2-1}
\d\theta\bigg) \|v_1-v_2\|_{\lambda} \\
&\le c(T/\lambda)^{\alpha/2} \|v_1-v_2\|_{\lambda} ,
\quad\forall\, v_1,v_2\in C([0,T];X )_\lambda .
\end{aligned}
\end{equation}
By choosing a sufficiently large $\lambda$,
the last inequality implies
 \begin{equation*}%\label{map-m-contr}
\|e^{-\lambda t}(Mv_1(t)-Mv_2(t))\|_{X }
\le \tfrac{1}{2} \|v_1-v_2\|_{\lambda} ,
\quad\forall\, v_1,v_2\in C([0,T];X )_\lambda .
\end{equation*}
Hence, the map $M$ is contractive on the space $C([0,T];X)_\lambda$.
The Banach fixed point theorem implies that $M$ has a unique fixed point,
which is also the unique solution of \eqref{re-form-nonlinear}.

\medskip

\noindent{\it Step 2: $C^\alpha([0,T];X)$ regularity.}
Consider the difference quotient for $h>0$
\begin{align}\label{Holder-u}
\begin{aligned}
\frac{u(t+h)-u(t)}{h^\alpha}
&=\frac{F(t+h)-F(t)}{h^\alpha} A  u_0
+\frac{1}{h^\alpha}\int_{t}^{t+h} E(s)Pf(u(t-s))\d s \\
&\quad + \int_0^{t} E(s)\frac{Pf(u(t+h-s))-Pf(u(t-s))}{h^\alpha}\d s =:\sum_{i=1}^3{\mathcal I}_i(t,h) . %+{\mathcal I}_2(t,h)+{\mathcal I}_3(t,h)  ,
\end{aligned}
\end{align}
A simple consequence of Lemma \ref{lem:smoothing}(i) is that
$h^{-\alpha} \|{F(t+h)-F(t)}\|_{X \rightarrow X } \le c ,$
which implies $\|{\mathcal I}_1(t,h)\|_{X }\le c$.
By appealing to  Lemma \ref{lem:smoothing}(iii), we have
\begin{align*}
\|{\mathcal I}_2(t,h)\|_{X }
&=\bigg\|\frac{1}{h^\alpha}\int_{t}^{t+h} E(s)Pf(u(t-s))\d s\bigg\|_{X } \\
&\le c\frac{1}{h^\alpha} \int_{t}^{t+h}s^{\alpha-1}\d s= \frac{c}{\alpha} \frac{(t+h)^\alpha-t^\alpha}{h^\alpha} \le c.
\end{align*}
By the Lipschitz continuity of $f$, we have
\begin{align*}
e^{-\lambda t}\|{\mathcal I}_3(t,h)\|_{X }
&=\bigg\|e^{-\lambda t}\int_0^{t} E(t-s)\frac{Pf(u(s+h))-Pf(u(s))}{h^\alpha}\d s\bigg\|_{X } \\
&\le c_1 \int_0^{t}e^{-\lambda (t-s)} (t-s)^{\alpha-1}e^{-\lambda s}\bigg\|\frac{u(s+h)-u(s)}{h^\alpha}\bigg\|_{X }\d s  .
\end{align*}
By substituting the estimates of ${\mathcal I}_i(t,h)$, $i=1,2,3$, into \eqref{Holder-u} and
denoting $ W_h(t)= e^{-\lambda t}h^{-\alpha}\|{u(t+h)-u(t)}\|_{X } ,$ we obtain
\begin{equation*}%\label{Wh2-1}
W_h(t)\le c+c_1\int_0^{t} e^{-\lambda (t-s)}(t-s)^{\alpha-1} W_h(s)\d s
\le c+c_1(T/\lambda)^{\frac{\alpha}{2}}
\max_{s\in[0,T]}W_h(s) ,
\end{equation*}
where the last inequality can be derived in the same way as \eqref{map-m-contr-1}. By choosing
a sufficiently large $\lambda$ and taking maximum of the left-hand side with respect to
$t\in[0,T]$, it implies $\displaystyle\max_{t\in[0,T]}W_h(t)\le c $, which further yields
$$h^{-\alpha}\|{u(t+h)-u(t)}\|_{X }\le ce^{\lambda t}\le c ,$$
where the constant $c$ is independent of $h$. Thus, we have proved $\|u\|_{C^\alpha([0,T];X )}
\le c.$\medskip

\noindent{\it Step 3: $C([0,T];D)$ regularity.} By applying the operator $A $ to both sides of
\eqref{re-form-nonlinear} and using the identity $AF(t)=\int_0^tAE(t-s)\d s$, cf. Lemma \ref{lem:smoothing},
we obtain
\begin{align} \label{Delta-u-reg}
\begin{aligned}
\quad A  u(t)-A  u_0
&=A  F(t)A  u_0+\int_0^t A  E(t-s) P  f(u(s)) \d s \\
&= A  F(t)\left(A  u_0+P f(u(t))\right)+\int_0^t A  E(t-s)   (Pf(u(s)-Pf(u(t)))\d s\\
&={\mathcal I}_4(t)+{\mathcal I}_5(t) .
\end{aligned}
\end{align}
By Lemma \ref{lem:smoothing}(iii) and the $C^\alpha([0,T];X)$ regularity from Step 2, we have
\begin{align*}
\|{\mathcal I}_5(t)\|_{X }
&=\bigg\|\int_0^t A  E(t-s)   (Pf(u(s))-Pf(u(t)))\d s\bigg\|_{X } \\
&\le \int_0^t \frac{c\|u(s)-u(t)\|_{X } }{t-s} \d s
\le \int_0^t \frac{c|t-s|^\alpha }{t-s} \d s
\le ct^\alpha ,\quad\forall\, t\in(0,T] .
\end{align*}
Lemma \ref{lem:smoothing}(iv) implies that ${\mathcal I}_5(t)$ is continuous for $t\in(0,T]$,
and the last inequality implies that ${\mathcal I}_5(t)$ is also continuous at $t=0$.
Hence ${\mathcal I}_5 \in C([0,T];X )$.
Moreover, Lemma \ref{lem:smoothing}(ii) gives ${\mathcal I}_4\in C([0,T];X )$ and
\begin{equation*}
\|{\mathcal I}_4(t)\|_X\le c\|A  u_0+P f(u(t))\|_X \le c.
\end{equation*}
Substituting the estimates of ${\mathcal I}_4(t)$ and ${\mathcal I}_5(t)$ into \eqref{Delta-u-reg} yields
$\|Au\|_{C([0,T];X )}\le c$, which further implies $\|u\|_{C([0,T];D )}\le c .$
The regularity result $u\in C([0,T];D )$ together with \eqref{nonlinear-abstr} yields
$\partial_t^\alpha u=A  u+Pf(u)\in C([0,T];X ) .$\medskip

\noindent{\it Step 4: Estimate of $\|u'(t)\|_X$.} By differentiating
\eqref{re-form-nonlinear} with respect to $t$, we obtain
\begin{align*}
\begin{aligned}
u'(t)
&=F'(t)A  u_0+E(t)  P f(u_0)  +\int_0^t E(s) P  f'(u(t-s))u'(t-s) \d s \\
&=E(t)(A  u_0+ P f(u_0) ) +\int_0^t E(t-s)  P f'(u(s))u'(s) \d s .
\end{aligned}
\end{align*}
By multiplying this equation by $t^{1-\alpha}$, we get
\begin{align*}
\begin{aligned}
t^{1-\alpha}u'(t)
&=t^{1-\alpha}E(t)(A  u_0+P  f(u_0) ) +\int_0^t t^{1-\alpha}s^{\alpha-1} E(t-s) P  f'(u(s)) s^{1-\alpha} u'(s) \d s ,
\end{aligned}
\end{align*}
which together with the $L^\infty$ stability of $P_h$ \cite[Lemma 6.1]{Thomee:2006} directly implies that
\begin{align*}
\begin{aligned}
e^{-\lambda t}t^{1-\alpha}\|u'(t)\|_{X}
&\le e^{-\lambda t}t^{1-\alpha}\|E(t)\|_{X \rightarrow X }
\|A  u_0+ P f(u_0) \|_{X } \\
&\quad +\int_0^t e^{-\lambda (t-s)}t^{1-\alpha}s^{\alpha-1} (t-s)^{\alpha-1} \|Pf'(u(s))\|_{L^\infty\II}
e^{-\lambda s}s^{1-\alpha} \|u'(s)\|_{X } ds \\
&\le ce^{-\lambda t}\|A  u_0+ P f(u_0) \|_{X }
 + c (T/\lambda)^{\frac{\alpha}{2}}  \max_{s\in[0,T]}e^{-\lambda s}s^{1-\alpha}\| u'(s)\|_{X }  .
\end{aligned}
\end{align*}
where the last line follows similarly as \eqref{map-m-contr-1}.
By choosing a sufficiently large $\lambda$ and taking maximum of the left-hand
side with respect to $t\in[0,T]$, it implies
$\displaystyle\max_{t\in[0,T]}\|e^{-\lambda t} t^{1-\alpha} u'(t)\|_{X }
\le c $, which further yields \eqref{reg-PDE2}.
The proof of Theorem \ref{THM:Reg} is complete.
\endproof

\section{Error estimates}\label{sec:error}

Now, we derive error estimates for the numerical solutions of
problem \eqref{nonlinear-PDE} using the discrete Gr\"onwall's inequality from Section
\ref{sec:Gronwall} and discrete maximal $\ell^p$-regularity from \cite{JLZ}.
To illustrate the general framework for the numerical analysis of nonlinear time fractional
diffusion equations, we focus on the L1 scheme and backward Euler CQ. Other time
stepping schemes can be analyzed similarly. The convergence rates we show below are
sharp (up to a logarithmic factor) with respect to the solution regularity in Theorem
\ref{THM:Reg}, and also confirmed by the numerical experiments in Section \ref{sec:numerics}.

\subsection{Preliminaries on the linear problem}
First we recall some error estimates for the following linear subdiffusion equation:
\begin{equation}\label{PDEv-linear}
\partial_t^\alpha v(t)-\Delta  v(t)=g(t), \quad\,\,\forall t\in (0,T] ,
\end{equation}
where $g$ is a given function.
The semidiscrete FEM for \eqref{PDEv-linear}
seeks $v_h(t)\in S_h$ such that
\begin{equation}\label{eqn:semidiscrete-linear}
\partial_t^\alpha v_h(t)-\Delta_h v_h(t)=P_hg(t), \quad\,\,\forall t\in (0,T],
\end{equation}
with $v_h(0)=R_hv(0)$, and the fully discrete scheme seeks $v_h^n\in S_h$, $n=1,\dots,N$, such that
\begin{equation}\label{eqn:fully-linear}
\bar\partial_\tau^\alpha (v_h^n-v_h^0)
-\Delta_h v_h^n= P_hg(t_n),
\end{equation}
with $v_h^0=v_h(0)$,
where $\bar\partial_\tau^\alpha v_h^n$ denotes either the backward Euler CQ or the L1 scheme.

The semidiscrete solution $v_h$ satisfies the following error estimate \cite{JinLazarovZhou:SIAM2013,
JinLazarovPasciakZhou:IMA2014,JinLazarovZhou:SISC2016}.

\begin{lemma}[Semidiscrete solution of linear problems]\label{lem:space-error-linear}
For the semidiscrete solution $v_h$ to problem \eqref{eqn:semidiscrete-linear},
there holds with $\ell_h=\log(2+1/h)$
\begin{equation*}
\max_{t\in[0,T]}
\|v_h(t)-v(t)\|_{L^2(\Omega)} \leq ch^2\|v(0)\|_{H^2(\Omega)}+ ch^2\ell_h^2 \|g\|_{L^\infty(0,T;L^2(\Omega))}.\\[-10pt]
\end{equation*}
\end{lemma}

The solution $v_h^n$ of the fully discrete scheme \eqref{eqn:fully-linear} satisfies
the following error estimate.
For the backward Euler CQ, it was proved in \cite[Theorems 3.5 and 3.6]{JinLazarovZhou:SISC2016},
while the proof for the L1 scheme will be given in Section \ref{app:time-error}.
\begin{lemma}[Fully discrete solutions of linear problems]\label{lem:time-error-linear}
For the fully discrete solutions $v_h^n$ to problem \eqref{eqn:fully-linear} with the L1
scheme or backward Euler CQ, there holds
\begin{equation*}
\begin{aligned}
%\max_{1\le n\le N}
\|v_h(t_n)-v_h^n\|_{L^2(\Omega)}
\le &\,  c\tau t_n^{\alpha-1}( \| \Delta v(0) \|_{L^2(\Omega)}+ \|g(0)\|_{L^2(\Omega)}) +
c\tau\int_0^{t_n}(t_n-s)^{\alpha-1}\|g'(s)\|_{L^2(\Omega)}\d s  . \\[-5pt]
\end{aligned}
\end{equation*}
\end{lemma}

\begin{remark}\label{rem:initial}{\upshape
If $1\le d\le 3$ and $v(0)\in H^1_0(\Omega)\cap H^2(\Omega)$, then the error estimates
in Lemmas \ref{lem:space-error-linear} and \ref{lem:time-error-linear}
are still valid if $v_h(0)$ is the Lagrange interpolation of $v(0)$,
due to the smoothing property of the solution operator \cite[Lemma 3.1]{JinLazarovZhou:SIAM2013}.
Consequently, all the results in Section \ref{ssub:nonlinear-error} remain valid
in this case.}
\end{remark}

Lemmas \ref{lem:space-error-linear} and \ref{lem:time-error-linear}
will be used below in the analysis of the nonlinear problem.

\subsection{Error estimates for the nonlinear problem}\label{ssub:nonlinear-error}
Now we can present error estimates for problem \eqref{nonlinear-PDE}.
Like in the linear case, we discuss the spatial error and temporal
error separately. First, we derive the spatial discretization error.

\begin{theorem}\label{THM:Error-1}
Let $u_0\in H_0^1(\Omega)\cap H^2(\Omega)$, and $f:{\mathbb R}\rightarrow {\mathbb R}$ be Lipschitz continuous.
Then the semidiscrete problem \eqref{nonlinear-FEM} has a unique solution
$u_h\in C([0,T];L_h^2(\Omega))$, which satisfies
\begin{align}\label{error-estimate-FEM}
 \max_{0\leq t\leq T}\|u(t)-u_h(t)\|_{L^2(\Omega)}\le c\ell_h^2 h^2 .
\end{align}
\end{theorem}
\begin{proof}
By Theorem \ref{THM:Reg}, the existence and uniqueness of the solution $u_h$ hold.
It remains to establish the estimate \eqref{error-estimate-FEM}.
To this end, we define $v_h(t)$ as the solution of
\begin{equation*}
    \partial_t^\al v_h(t) - \Delta_h v_h(t) = P_h f(u(t)), \quad \text{with}\quad v_h(0)=u_h(0)=R_h u_0.
\end{equation*}
This together with Lemma \ref{lem:space-error-linear} yields the following estimate for $t \ge 0$
\begin{equation}\label{eqn:wh}
\begin{split}
\| (u-v_h)(t) \|_{L^2(\Omega)}
%&\le ch^2 \| u(0)  \|_{H^2(\Omega)} + ch^2 \ell_h^2 \|  f(u) - P_h f(u) \|_{L^\infty(0,T;L^2\II)}\\
    \le ch^2 \| u(0)  \|_{H^2(\Omega)} + ch^2 \ell_h^2 \|  f(u) \|_{L^\infty(0,T;L^2\II)} \le ch^2\ell_h^2.
\end{split}
\end{equation}
Meanwhile, we note that $\rho_h:=v_h-u_h$ satisfies the following equation
\begin{equation*}
    \partial_t^\al \rho_h(t) -\Delta \rho_h(t) = P_h f (u(t)) - P_h f(u_h(t)), \quad \text{with}\quad \rho_h(0)=0.
\end{equation*}
Then, by the Lipschitz continuity of $f$ and the maximal $L^p$-regularity of fractional
evolution equations \cite[Corollary 1]{Bazhlekova:2002}, we obtain the following estimate
for any $p\in(1,\infty)$
\begin{equation*}
\begin{split}
    \|  \partial_t^\al \rho_h  \|_{L^p(0,T; L^2\II)}
     &\le c \|  P_h f (u ) - P_h f(u_h ) \|_{L^p(0,T; L^2\II)}\\
     & \le c \| u-u_h \|_{L^p(0,T;L^2\II)}\\
     &\le c\| u-v_h \|_{L^p(0,T;L^2\II)} + c \| \rho_h \|_{L^p(0,T;L^2\II)}\\
     &\le c h^2\ell_h^2 +c\| \rho_h \|_{L^p(0,T;L^2\II)}  .
\end{split}
\end{equation*}
Then by the fractional Gr\"onwall's inequality in Theorem \ref{Frac-Gronwall}, we have
\begin{equation*}
\max_{t\in[0,T]}\| \rho_h(t) \|_{L^2\II} \le c h^2\ell_h^2.
\end{equation*}
This and \eqref{eqn:wh} directly imply the desired result.
\end{proof}

Next we give the temporal discretization error.
\begin{theorem}\label{THM:Error-2}
Let $u_0\in H_0^1(\Omega)\cap H^2(\Omega)$, and $f:{\mathbb R}\rightarrow {\mathbb R}$ be Lipschitz continuous. Then the fully discrete scheme \eqref{TD-scheme}, with either the L1
scheme or backward Euler CQ for time discretization, has a unique solution $u_h^n\in S_h$,
$n=1,\dots,N$, and the solutions satisfy
\begin{equation}\label{error-estimate}
\max_{1\le n\le N}  \|u_h(t_n)-u_h^n\|_{L^2(\Omega)}\le c\tau^\alpha .
\end{equation}
\end{theorem}
\begin{proof}
For given $u_h^0, \cdots, u_h^{n-1}$, \eqref{TD-scheme}
is essentially a linear system with a symmetric positive definite matrix,
and thus it has a unique solution $u_h^n\in S_h$.
It suffices to establish the estimate \eqref{error-estimate}.
Like before, we decompose the fully discrete solution $u_h^n$ into two parts,
$u_h^n=v_h^n+\rho_h^n$, where $v_h^n$ and $\rho_h^n$ respectively satisfy
\begin{align}
& \bar\partial_\tau^\alpha( v_h^n-v_h^0)
 -\Delta_h v_h^n=P_hf(u_h(t_n)) , \label{FEM-vhn} \\
&\bar\partial_\tau^\alpha \rho_h^n -\Delta_h \rho_h^n
=P_hf(u_h^{n-1})-P_hf(u_h(t_n)), \label{PDE-vn}
\end{align}
with $v_h^0=u_h(0)=R_h u_0$  and $\rho_h^0=0$.
Equation \eqref{FEM-vhn} can be viewed as the time discretization of \eqref{nonlinear-FEM},
with the right-hand side being a given function. Hence,
by Lemma \ref{lem:time-error-linear} and using $\|\partial_su_h(s)\|_{L^2(\Omega)}\le cs^{\alpha-1}$
(cf. Theorem \ref{THM:Reg}) and Rademacher's theorem, we have
\begin{align}\label{estimate-u-w}
\begin{aligned}
\|u_h(t_n)-v_h^n\|_{L^2(\Omega)}
\le& ct_n^{\alpha-1}\tau\bigg(\|\Delta_h u_h(0)\|_{L^2(\Omega)}+\|f(u_h(0))\|_{L^2(\Omega)}\bigg)\\
   &+c\tau\int_0^{t_n}(t_n-s)^{\alpha-1}\|f'(u_h(s))\partial_su_h(s)\|_{L^2(\Omega)}\d s \\
\le & ct_n^{\alpha-1}\tau +c\tau\int_0^{t_n}(t_n-s)^{\alpha-1} s^{\alpha-1} \d s  \\
\le & ct_n^{\alpha-1}\tau +ct_n^{2\alpha-1}\tau
\le  c\tau^\alpha .
\end{aligned}
\end{align}
It remains to estimate $\rho_h^n$. By applying the discrete maximal $\ell^p$-regularity
to \eqref{PDE-vn} (choosing $X=L_h^2(\Omega)$ in \cite[Theorems 3.1 and 4.1]{JLZ}),
we obtain that for all $1<p<\infty$:
\begin{equation*}
  \begin{aligned}
    \|(\bar\partial_\tau^\alpha \rho_h^n)_{n=1}^m\|_{\ell^p(L^2(\Omega))}
&\le c\|(f(u_h^{n-1})-f(u_h(t_n)))_{n=1}^m\|_{\ell^p(L^2(\Omega))} \\
&\le c\|(f(u_h^{n-1})-f(u_h(t_{n-1})))_{n=1}^m\|_{\ell^p(L^2(\Omega))}\\
&\quad +c\|(f(u_h(t_{n-1}))-f(u_h(t_n)))_{n=1}^m\|_{\ell^p(L^2(\Omega))}.
  \end{aligned}
\end{equation*}
By the Lipschitz continuity of $f$ and the triangle inequality, we arrive at
\begin{equation*}
  \begin{aligned}
   &\quad\|(f(u_h^{n-1})-f(u_h(t_{n-1})))_{n=1}^m\|_{\ell^p(L^2(\Omega))}\\
    &\leq c\|(u_h(t_{n-1})-u_h^{n-1})_{n=1}^m\|_{\ell^p(L^2(\Omega))} \\
   & \leq  c\|(u_h(t_{n-1})-v_h^{n-1})_{n=1}^m\|_{\ell^p(L^2(\Omega))} +c\|( \rho_h^{n-1})_{n=1}^m\|_{\ell^p(L^2(\Omega))} \\
  & \le c\tau^\alpha +c\|(\rho_h^{n})_{n=1}^{m-1}\|_{\ell^p(L^2(\Omega))} ,
  \end{aligned}
\end{equation*}
where the last inequality follows from \eqref{estimate-u-w}. Similarly, by the Lipschitz
continuity of $f$ and the a priori estimate $\|u_h\|_{C^\alpha([0,T];L^2(\Omega))}\le c$ (cf. Theorem \ref{THM:Reg}), we deduce
\begin{equation*}
  \begin{aligned}
   \quad \|(\|f(u_h(t_{n-1}))-f(u_h(t_n))\|_{L^2(\Omega)})_{n=1}^m\|_{\ell^p}
   &\leq c\|(\|u_h(t_{n-1})-u_h(t_n)\|_{L^2(\Omega)})_{n=1}^m\|_{\ell^p} \\
   &\leq c\|(c\tau^\alpha)_{n=1}^m\|_{\ell^p}.
  \end{aligned}
\end{equation*}
Combining the preceding three estimates yields
\begin{equation*}
    \|(\bar\partial_\tau^\alpha \rho_h^n)_{n=1}^m\|_{\ell^p(L^2(\Omega))} \le c\|(\rho_h^{n})_{n=1}^{m-1}\|_{\ell^p(L^2(\Omega))}
+c\tau^\alpha .
\end{equation*}
By choosing $p>1/\alpha$ and applying the discrete Gr\"onwall's inequality (with
$X=L^2(\Omega)$ in Theorem \ref{L1-CQ-gronwall}), we obtain
\begin{align}\label{estimate-v}
\max_{1\le n\le N}\|\rho_h^n\|_{L^2(\Omega)}\le c\tau^\alpha .
\end{align}
In view of the decomposition $u_h(t_n)-u_h^n=(u_h(t_n)-v_h^n) - \rho_h^n$,
the two estimates \eqref{estimate-u-w} and \eqref{estimate-v} imply \eqref{error-estimate},
completing the proof of the theorem.
\end{proof}

\begin{remark}\label{Error-non-Lipschitz}
{\upshape
If the nonlinear source $f$ is not Lipschitz continuous but problem \eqref{nonlinear-PDE} has
a unique bounded solution $u$, then Theorems \ref{THM:Error-1} and \ref{THM:Error-2} are still
valid by proving the boundedness of the semidiscrete solution $u_h$ and the fully discrete solution $u_h^n$.
For simplicity, we have assumed $f$ to be Lipschitz continuous in order to avoid these technicalities.
}\end{remark}

\section{Proof of Lemma \ref{lem:time-error-linear} for the L1 scheme}\label{app:time-error}
The L1 scheme was analyzed in \cite{JinLazarovZhou:2016ima} only for the homogeneous problem.
Below we give a proof for the general case.

First, we assume that $g$ is time-independent,
i.e., $g(t) \equiv g(0)$. Then using Laplace transform, one can derive the following
error representation (cf. \cite[eq. (2.7) and (2.9)]{JinLazarovZhou:2016ima}):
\begin{equation*}
\begin{aligned}
v_h(t_n)-v_h^n
  &= \frac{1}{2\pi\mathrm{i}}\int_{\Gamma_{\theta,\delta}}e^{zt_n}z^{-1}(z^\alpha - \Delta_h)^{-1}(\Delta_hv_h(0)+P_hg(0)) \d z \\
  &\quad
  - \frac{1}{2\pi\mathrm{i} }\int_{\Gamma_{\theta,\delta}^\tau }e^{zt_{n}}\mu(e^{-z\tau})^{-1} ( \beta_\tau(e^{-z\tau}) - \Delta_h)^{-1}(\Delta_hv_h(0)+P_hg(0))  \,\d z ,
  \end{aligned}
\end{equation*}
where the contour $\Gamma_{\theta,\delta}$ is defined in \eqref{contour-Gamma}, $\Gamma_{\theta,\delta}^\tau
=\{z\in \Gamma_{\theta,\delta}:|{\rm Im}(z)|\le 1/\tau\}$, and
\begin{equation*}
    \mu(z)=\frac{1-e^{-z\tau}}{\tau e^{-z\tau}} \quad \text{and}\quad
    \beta_\tau(e^{-z\tau}) = \frac{(1-e^{-z\tau})^2}{e^{-z\tau}\tau^\al\Gamma(2-\al)} \mathrm{Li}_{\alpha-1}(e^{-z\tau}),
\end{equation*}
which satisfy the following estimates (cf. \cite[Section 3]{JinLazarovZhou:2016ima}):
\begin{align}
    &  c_0|z| \leq  |\mu(e^{-z\tau})| \leq c_1 |z| \quad \mbox{and}\quad
    |\mu(e^{-z\tau})-z|\leq c\tau |z|^2, \quad \forall z\in \Gamma_{\theta,\delta}^\tau , \label{eqn:est-g-2}\\
   & |\beta_\tau(e^{-z\tau})| \ge c|z|\tau^{1-\al} \quad\mbox{and}\quad
        |\beta_\tau(e^{-z\tau})-z^\alpha| \le c|z|^2\tau^{2-\al},\quad \forall z\in \Gamma_{\theta,\delta}^\tau.\label{eqn:est-g-4}
\end{align}
By using \eqref{eqn:est-g-2}--\eqref{eqn:est-g-4}, direct calculations yield
\begin{equation}\label{eqn:es0}
\| z^{-1}(z^\alpha - \Delta_h)^{-1} - \mu(e^{-z\tau})^{-1}(\beta_\tau(e^{-z\tau}) - \Delta_h)^{-1} \|_{L^2\II\rightarrow L^2\II} \le c |z|^{-\al} \tau.
\end{equation}
Now we split the error $v_h(t_n)-v_h^n$ into two components,  i.e., $v_h(t_n)-v_h^n=\mathcal{I}_1+\mathcal{I}_2$, where
\begin{align*}
% &v_h(t_n)-v_h^n\\
\mathcal{I}_1 & =  \frac{1}{2\pi\mathrm{i}} \int_{\Gamma_{\theta,\delta}^\tau } e^{zt_n}
  \left(  z^{-1}(z^\alpha - \Delta_h)^{-1}-\mu(e^{-z\tau})^{-1} (\beta_\tau(e^{-z\tau}) - \Delta_h)^{-1}\right) (\Delta_hv_h(0)+P_hg(0)) \,\d z,\\
\mathcal{I}_2 &=\frac{1}{2\pi\mathrm{i}} \int_{\Gamma_{\theta,\delta}\backslash\Gamma_{\theta,\delta}^\tau } e^{zt_n} z^{-1}(z^\alpha - \Delta_h)^{-1}(\Delta_hv_h(0)+P_hg(0)) \d z.
\end{align*}
By using \eqref{eqn:es0} and \eqref{eqn:resol}, and choosing $\delta \le 1/t_n$,
the argument from \cite{JinLazarovZhou:2016ima} yields
\begin{align}\label{vhtn-vhn}
   & \| \mathcal{I}_1 \|_{L^2(\Omega)}
   +\| \mathcal{I}_2 \|_{L^2(\Omega)} \le ct_n^{ \al-1}\tau \| \Delta_hv_h(0)+P_hg(0)\|_{L^2(\Omega)}  .
\end{align}

Second, we consider the case $v(0)= g(0)=0$. Then Taylor's expansion gives
\begin{equation}\label{eqn:f}
    P_hg(t) =P_hg(0) + 1*P_hg'(t) =  1*P_hg'(t) .
\end{equation}
In view of \eqref{Int-Eq-w}, the semidiscrete solution $v_h(t_n)$ can be represented by
\begin{align}\label{repr-vhtn}
  v_h(t_n) &= (E*P_hg)(t_n)=(E*(1*P_hg'))(t_n)  =((E*1)*P_hg')(t_n) .
\end{align}
Similarly, we have
$$
({\beta_\tau(\xi)}-\Delta_h)^{-1} = \sum_{n=0}^\infty E_\tau^n \xi^n\quad\mbox{with}\quad
E_\tau^n = \frac{\tau}{2\pi\mathrm{i}}\int_{\Gamma_{\theta,\delta}^\tau } e^{zn\tau} ({ \beta_\tau(e^{-z\tau})}-\Delta_h)^{-1}\,\d z .
$$
Hence the fully discrete solution $v_h^n$ can be represented
by $v_h^n = \sum_{j=0}^n  E_\tau^{n-j}P_hg(t_j),$
and the second inequality of \eqref{eqn:est-g-4} implies
\begin{equation}\label{eqn:Etaun}
    \|  E_\tau^n  \|_{L^2(\Omega)\rightarrow L^2(\Omega)} \le ct_n^{\alpha-1}\tau   .
\end{equation}
Let $E_{\tau,\epsilon}(t) = \sum_{ n=0}^\infty E_\tau^{n}\delta_{t_n-\epsilon}(t) $,
where $\delta_{t_n-\epsilon}$ is the Dirac--Delta function concentrated at $t_n-\epsilon$, with $\epsilon\in(0,\tau)$.
Then $ v_h^n$ can be rewritten as
\begin{equation}\label{repr-vhn}
  v_h^n =\lim_{\epsilon\rightarrow 0} (E_{\tau,\epsilon} * P_hg) (t_n)
  =\lim_{\epsilon\rightarrow 0} (E_{\tau,\epsilon}*(1*P_hg'))(t_n)
  = (\lim_{\epsilon\rightarrow 0}(E_{\tau,\epsilon}*1)*P_hg')(t_n).
\end{equation}
The representations \eqref{repr-vhtn} and \eqref{repr-vhn} yield
\begin{equation}\label{vh-vdtg}
 \| v_h(t_n)-v_h^n \|_{L^2(\Omega)}
 \le \|[\lim_{\epsilon\rightarrow 0}((E-E_{\tau,\epsilon})*1)*P_hg'](t_n)\|_{L^2(\Omega)} .
\end{equation}
Using Laplace transform and Cauchy's integral formula, we deduce
\begin{align*}
    (\lim_{\epsilon\rightarrow 0} (E-E_{\tau,\epsilon}) *1) (t_n)
    &= \frac{1}{2\pi{\rm i}} \int_{\Gamma_{\theta,\delta}}e^{zt_n}z^{-1}(z^\alpha-\Delta_h)^{-1}\,\d z\\
     &\quad -\frac{1}{2\pi\mathrm{i}}\int_{\Gamma_{\theta,\delta}^\tau }  e^{zt_{n}}\mu(e^{-z\tau})^{-1} ( \beta_\tau(e^{-z\tau}) - \Delta_h)^{-1}\d z.
\end{align*}
Then using the estimate \eqref{eqn:es0} we obtain
\begin{equation}\label{E-Etautn}
\| (\lim_{\epsilon\rightarrow 0}(E-E_{\tau,\epsilon})*1)(t_n)  \|_{L^2(\Omega)\rightarrow L^2(\Omega)} \le c  \tau t_n^{\alpha-1} .
\end{equation}
It remains to prove the following extension of the estimate \eqref{E-Etautn}:
\begin{equation}\label{eqn:tn}
 \| (\lim_{\epsilon\rightarrow 0}(E-E_{\tau,\epsilon})*1)(t) \|_{L^2(\Omega)\rightarrow L^2(\Omega)} \le c\tau t^{\alpha-1},\quad
 \forall\, t\in(0,T) .
\end{equation}
Then this and \eqref{vh-vdtg} yield
the second part on the right-hand side of \eqref{lem:time-error-linear},
and completes the proof of Lemma \ref{lem:time-error-linear}.

To prove \eqref{eqn:tn}, we consider the Taylor expansion of $(E(t)-E_{\tau,\epsilon}(t))\ast 1$ at $t=t_n$, i.e.,
\begin{equation}\label{E-EtauTaylor}
   ( (E-E_{\tau,\epsilon})*1)(t) = ((E-E_{\tau,\epsilon})*1)(t_n)
    - \int_t^{t_n} (E-E_{\tau,\epsilon})(s)\,\d s .
\end{equation}
In view of Lemma \ref{lem:smoothing} (iii), there holds
\begin{equation*}
    \bigg\| \int_t^{t_n} E(s)\,\d s \bigg \|_{L^2(\Omega)\rightarrow L^2(\Omega)}
    \le c \int_t^{t_n} s^{\al-1} \,\d s \le c\tau t^{\alpha-1} .
\end{equation*}
Similarly, appealing to \eqref{eqn:Etaun}, we have
\begin{equation*}
    \bigg\|\lim_{\epsilon\rightarrow 0}\int_t^{t_n} E_{\tau,\epsilon}(s)\,\d s \bigg\|_{L^2(\Omega)\rightarrow L^2(\Omega)}
    =\|E_{\tau}^n\|_{L^2(\Omega)\rightarrow L^2(\Omega)}    \le ct_n^{\alpha-1}\tau .
\end{equation*}
Substituting \eqref{E-Etautn}  and the last two inequalities into \eqref{E-EtauTaylor} yields \eqref{eqn:tn}.
\endproof

\section{Numerical experiments}\label{sec:numerics}
In this section, we present numerical examples to verify the theoretical
results in Theorems \ref{THM:Error-1} and \ref{THM:Error-2}.
We consider problem \eqref{nonlinear-PDE} with a
diffusion coefficient $0.1$ in the unit square $\Omega=(0,1)^2$, with the following two sets of problem data:
\begin{itemize}
\item[(a)] $u_0(x,y)=xy(1-x)(1-y)\in H^1_0(\Omega)\cap H^2(\Omega)$ and $f=\sqrt{1+u^2}$;
\item[(b)] $u_0(x,y)=x(1-x)\sin(2\pi y)\in H^1_0(\Omega)\cap H^2(\Omega)$ and $f=1-u^3$.
\end{itemize}

In the computation, we divided the domain $\Omega$ into regular right triangles with $M$ equal subintervals of length $h=1/M$
on each side of the domain. The numerical solutions are computed by using the Galerkin FEM in space,
and the backward Euler (BE) CQ or the L1 scheme in time.  To evaluate the convergence, we compute
the spatial error $e_t$ and temporal error $e_s$, respectively, defined by
\begin{equation*}
 e_s = \max_{1\leq n\leq N}\|u_h(t_n)-u(t_n)\|_{L^2(\Omega)}
 \quad \mbox{and}\quad
  e_t = \max_{1\le n\le N}\|u_h^n-u_h(t_n)\|_{L^2(\Omega)}.
\end{equation*}
Since the exact solution to problem \eqref{nonlinear-PDE} is unavailable, we compute reference
solutions on a finer mesh, i.e., the continuous solution $u(t_n)$ with a fixed time step $\tau=1/1000$
and mesh size $h=1/1280$, and the semidiscrete solution $u_h(t_n)$ with $h=1/10$ and
$\tau=1/(64\times10^4)$.

In case (a), since the nonlinearity $f$ is Lipschitz continuous, the theory in
Section \ref{sec:error} applies. The numerical results
for case (a) are shown in Tables \ref{tab:a-space} and \ref{tab:a-time}, where
the numbers in the bracket in the last column refer to the theoretical predictions
from Section \ref{sec:error}. We observe an $O(h^2)$ rate for
the spatial error $e_s$, and an $O(\tau^\alpha)$ rate for the temporal
error $e_t$ for both backward Euler CQ and L1 scheme. These observations fully confirm
Theorems \ref{THM:Error-1} and \ref{THM:Error-2}.

\begin{table}[htb!]
\caption{Numerical results for case (a): the spatial error $ e_s$
with $T=1$, with  $N=1000$, $h=1/M$.}\label{tab:a-space}
\begin{center}
\vspace{-.3cm}{\setlength{\tabcolsep}{7pt}
     \begin{tabular}{|c|ccccc|c|}
     \hline
      $\alpha\backslash M$    &$5 $ &$10 $ & $20 $ & $40 $ &$80 $& rate \\
     \hline
       $0.4$       &6.89e-2 &2.00e-2 &5.34e-3 &1.37e-3 &3.31e-4  &  $\approx$ 2.01 (2.00)\\
       $0.6$       &7.06e-2 &2.05e-2 &5.58e-3 &1.42e-3 &3.44e-4  &  $\approx$ 2.01 (2.00)\\
       $0.8$       &7.59e-2 &2.18e-2 &5.80e-3 &1.48e-3 &3.57e-4  &  $\approx$ 2.01 (2.00)\\
      \hline
     \end{tabular}}
\end{center}
\end{table}

\begin{table}[htb!]
\caption{Numerical results for case (a): the temporal error $e_t$
with $T= 1$,  $\tau=T/N$,  $N=k\times10^4$, and $h=0.1$.}\label{tab:a-time}
\begin{center}
\vspace{-.3cm}{\setlength{\tabcolsep}{7pt}
     \begin{tabular}{|c|c|ccccc|c|}
     \hline
      $\alpha$ &  $k$    &$1 $ &$2 $ & $4 $ & $8 $ &$16$  &rate \\
     \hline
         $0.4$          & BE  &1.16e-3 &8.88e-4 &6.79e-4 &5.19e-4 &3.86e-4 & $\approx$ 0.39 (0.40)\\
                         &  L1  &2.06e-3 &1.59e-3 &1.22e-3 &9.34e-4 &7.15e-4 &  $\approx$ 0.38 (0.40)\\
      \hline
         $0.6$          & BE   &1.79e-4 &1.18e-4  &7.75e-5  &5.10e-5 &3.36e-5 & $\approx$ 0.60 (0.60)\\
                        &  L1   &3.05e-4 &2.02e-4 &1.33e-4 &8.80e-5  &5.81e-5  & $\approx$ 0.60 (0.60)\\
      \hline
        $0.8$           & BE    &1.73e-5  &9.87e-6 &5.65e-6 &3.24e-6 &1.86e-6 &  $\approx$ 0.80 (0.80)\\
                         & L1    &3.91e-5 &2.24e-5 &1.29e-5 &7.38e-6 &4.24e-6 &  $\approx$ 0.80 (0.80)\\
      \hline
     \end{tabular}}
\end{center}
\end{table}

In case (b), the nonlinear source $f$ is not Lipschitz continuous. Nonetheless, one observes
an $O(h^2)$ and $O(\tau^\alpha)$ convergence rate for the spatial and temporal errors, respectively,
cf. Tables \ref{tab:b-space} and \ref{tab:b-time}. This concurs with the discussions in
Remarks \ref{Reg-non-Lipschitz} and \ref{Error-non-Lipschitz}. Further, the absolute accuracy
of the L1 scheme and backward Euler CQ is comparable with each other for both cases (a) and (b).
Interestingly, the spatial error $e_s$ increases slightly with the fractional order $\alpha$,
but the temporal error $e_t$ decreases with $\alpha$.

\begin{table}[htb!]
\caption{Numerical results for case (b): the spatial error $e_s$
with $T=1$, with  $N=1000$, $h=1/M$.}\label{tab:b-space}
\begin{center}
\vspace{-.3cm}{\setlength{\tabcolsep}{7pt}
     \begin{tabular}{|c|ccccc|c|}
     \hline
      $\alpha \backslash M$    &$5 $ &$10 $ & $20 $ & $40 $ &$80 $& rate \\
     \hline
       $0.4$       &5.65e-2 &1.68e-2 &4.58e-3 &1.18e-3 &2.87e-4  &  $\approx$ 2.00 (2.00)\\
       $0.6$       &5.90e-2 &1.75e-2 &4.74e-3 &1.22e-3 &2.97e-4  &  $\approx$ 2.00 (2.00)\\
       $0.8$       &6.19e-2 &1.82e-2 &4.93e-3 &1.27e-3 &3.08e-4  &  $\approx$ 2.01 (2.00)\\
      \hline
     \end{tabular}}
\end{center}
\end{table}

\begin{table}[htb!]
\caption{Numerical results for case (b): the temporal error $e_t$
with $T=1$,  $\tau=T/N$, $N=k\times10^4$, $h=0.1$.}\label{tab:b-time}
\begin{center}
\vspace{-.3cm}{\setlength{\tabcolsep}{7pt}
     \begin{tabular}{|c|c|ccccc|c|}
     \hline
      $\alpha$ &  $k$    &$1 $ &$2 $ & $4 $ & $8 $ &$16 $  &rate \\
     \hline
         $0.4$          & BE  &1.53e-3 &1.17e-3 &9.07e-4 &6.96e-4 &5.33e-4  & $\approx$ 0.38 (0.40)\\
                         &  L1  &2.73e-3 &2.12e-3 &1.64e-3 &1.26e-3 &9.65e-4  &  $\approx$ 0.38 (0.40)\\
      \hline
         $0.6$          & BE   &2.43e-4 &1.60e-4 &1.05e-4 &6.93e-5  &4.56e-5  & $\approx$ 0.60 (0.60)\\
                        &  L1   &4.14e-4 &2.74e-4 &1.81e-4 &1.20e-4  &7.89e-5  & $\approx$ 0.60 (0.60)\\
      \hline
        $0.8$           & BE    &2,35e-5 &1.34e-5 &7.68e-6 &4.40e-6 &2.53e-6   &  $\approx$ 0.80 (0.80)\\
                         & L1    &5.30e-5 &3.04e-5 &1.75e-5 &1.00e-5 &5.76e-6  &  $\approx$ 0.80 (0.80)\\
      \hline
     \end{tabular}}
\end{center}
\end{table}

\section*{Acknowledgements}
The authors are grateful to the anonymous referees for their constructive comments, which are very helpful to improve the presentation of the paper.

\bibliographystyle{siam}
\bibliography{frac_nonlinear}

\begin{thebibliography}{10}

\bibitem{AkrivisLi2017}
{\sc G.~Akrivis and B.~Li}, {\em {Maximum norm analysis of implicit-explicit
  backward difference formulae for nonlinear parabolic equations}}, IMA J.
  Numer. Anal., DOI: 10.1093/imanum/drx008.

\bibitem{AkrivisLiLubich2016}
{\sc G.~Akrivis, B.~Li, and C.~Lubich}, {\em {Combining maximal regularity and
  energy estimates for time discretizations of quasilinear parabolic
  equations}}, Math. Comp., 86 (2017), pp.~1527--1552.

\bibitem{Alikhanov:2015}
{\sc A.~A. Alikhanov}, {\em A new difference scheme for the time fractional
  diffusion equation}, J. Comput. Phys., 280 (2015), pp.~424--438.

\bibitem{ABHN}
{\sc W.~Arendt, C.~J. Batty, M.~Hieber, and F.~Neubrander}, {\em {Vector-valued
  Laplace Transforms and Cauchy Problems}}, Birkh\"auser, Basel, 2nd~ed., 2011.

\bibitem{Bazhlekova:2002}
{\sc E.~Bazhlekova}, {\em Strict {$L^p$} solutions for fractional evolution
  equations}, Fract. Calc. Appl. Anal., 5 (2002), pp.~427--436.

\bibitem{Berkowitz:2002}
{\sc B.~Berkowitz, J.~Klafter, R.~Metzler, and H.~Scher}, {\em Physical
  pictures of transport in heterogeneous media: Advection-dispersion,
  random-walk, and fractional derivative formulations}, Water Res. Research, 38
  (2002), pp.~9--1--9--12.

\bibitem{Blunck:2001}
{\sc S.~Blunck}, {\em Maximal regularity of discrete and continuous time
  evolution equations}, Studia Math., 146 (2001), pp.~157--176.

\bibitem{ChenThomeeWalbin:1992}
{\sc C.~Chen, V.~Thom{\'e}e, and L.~B. Wahlbin}, {\em Finite element
  approximation of a parabolic integro-differential equation with a weakly
  singular kernel}, Math. Comp., 58 (1992), pp.~587--602.

\bibitem{CreedonMiller:1975}
{\sc D.~M. Creedon and J.~J.~H. Miller}, {\em The stability properties of
  {$q$}-step backward difference schemes}, BIT, 15 (1975), pp.~244--249.

\bibitem{CuestaLubichPalencia:2006}
{\sc E.~Cuesta, C.~Lubich, and C.~Palencia}, {\em Convolution quadrature time
  discretization of fractional diffusion-wave equations}, Math. Comp., 75
  (2006), pp.~673--696.

\bibitem{Flajolet:1999}
{\sc P.~Flajolet}, {\em Singularity analysis and asymptotics of {B}ernoulli
  sums}, Theoret. Comput. Sci., 215 (1999), pp.~371--381.

\bibitem{HairerNorsettWanner:2010}
{\sc E.~Hairer, S.~P. N{\o}rsett, and G.~Wanner}, {\em Solving {O}rdinary
  {D}ifferential {E}quations. {I}}, Springer-Verlag, Berlin, second~ed., 2010.
\newblock Nonstiff problems.

\bibitem{JinLazarovPasciakZhou:IMA2014}
{\sc B.~Jin, R.~Lazarov, J.~Pasciak, and Z.~Zhou}, {\em Error analysis of
  semidiscrete finite element methods for inhomogeneous time-fractional
  diffusion}, IMA J. Numer. Anal., 35 (2015), pp.~561--582.

\bibitem{JinLazarovZhou:SIAM2013}
{\sc B.~Jin, R.~Lazarov, and Z.~Zhou}, {\em Error estimates for a semidiscrete
  finite element method for fractional order parabolic equations}, SIAM J.
  Numer. Anal., 51 (2013), pp.~445--466.

\bibitem{JinLazarovZhou:2016ima}
\leavevmode\vrule height 2pt depth -1.6pt width 23pt, {\em An analysis of the
  {L}1 scheme for the subdiffusion equation with nonsmooth data}, IMA J. Numer.
  Anal., 36 (2016), pp.~197--221.

\bibitem{JinLazarovZhou:SISC2016}
\leavevmode\vrule height 2pt depth -1.6pt width 23pt, {\em Two fully discrete
  schemes for fractional diffusion and diffusion-wave equations with nonsmooth
  data}, SIAM J. Sci. Comput., 38 (2016), pp.~A146--A170.

\bibitem{JLZ}
{\sc B.~Jin, B.~Li, and Z.~Zhou}, {\em Discrete maximal regularity of
  time-stepping schemes for fractional evolution equations}.
\newblock Preprint, arXiv:1606.07587 (to appear in {\em Numer. Math.}).

\bibitem{Kemmochi:2015}
{\sc T.~Kemmochi}, {\em Discrete maximal regularity for abstract {C}auchy
  problems}, Studia Math., 234 (2016), pp.~241--263.

\bibitem{KilbasSrivastavaTrujillo:2006}
{\sc A.~A. Kilbas, H.~M. Srivastava, and J.~J. Trujillo}, {\em Theory and
  {A}pplications of {F}ractional {D}ifferential {E}quations}, Elsevier Science
  B.V., Amsterdam, 2006.

\bibitem{Kou:2008}
{\sc S.~Kou}, {\em Stochastic modeling in nanoscale biophysics: Subdiffusion
  within proteins}, Ann. Appl. Stat., 2 (2008), pp.~501--535.

\bibitem{KovacsLiLubich2016}
{\sc B.~Kov\'acs, B.~Li, and C.~Lubich}, {\em A-stable time discretizations
  preserve maximal parabolic regularity}, SIAM J. Numer. Anal., 54 (2016),
  pp.~3600--3624.

\bibitem{KLC2016}
{\sc P.~C. Kunstmann, B.~Li, and C.~Lubich}, {\em {Runge-Kutta time
  discretization of nonlinear parabolic equations studied via discrete maximal
  parabolic regularity}}, Preprint, arXiv:1606.03692 (to appear in {\em Found.
  Comput. Math.}).

\bibitem{KunstmannWeis:2004}
{\sc P.~C. Kunstmann and L.~Weis}, {\em Maximal {$L_p$}-regularity for
  parabolic equations, {F}ourier multiplier theorems and
  {$H^\infty$}-functional calculus}, in Functional {A}nalytic {M}ethods for
  {E}volution {E}quations, vol.~1855 of Lecture Notes in Math., Springer,
  Berlin, 2004, pp.~65--311.

\bibitem{LanglandsHenry:2005}
{\sc T.~A.~M. Langlands and B.~I. Henry}, {\em The accuracy and stability of an
  implicit solution method for the fractional diffusion equation}, J. Comput.
  Phys., 205 (2005), pp.~719--736.

\bibitem{LeykekhmanVexler:2017}
{\sc D.~Leykekhman and B.~Vexler}, {\em Discrete maximal parabolic regularity
  for {G}alerkin finite element methods}, Numer. Math., 135 (2017),
  pp.~923--952.

\bibitem{LinXu:2007}
{\sc Y.~Lin and C.~Xu}, {\em Finite difference/spectral approximations for the
  time-fractional diffusion equation}, J. Comput. Phys., 225 (2007),
  pp.~1533--1552.

\bibitem{Lubich:1986}
{\sc C.~Lubich}, {\em Discretized fractional calculus}, SIAM J. Math. Anal., 17
  (1986), pp.~704--719.

\bibitem{Lubich:1988}
\leavevmode\vrule height 2pt depth -1.6pt width 23pt, {\em Convolution
  quadrature and discretized operational calculus. {I}}, Numer. Math., 52
  (1988), pp.~129--145.

\bibitem{Luchko:2013}
{\sc Y.~Luchko, W.~Rundell, M.~Yamamoto, and L.~Zuo}, {\em Uniqueness and
  reconstruction of an unknown semilinear term in a time-fractional
  reaction–diffusion equation}, Inverse Problems, 29 (2013), p.~065019.

\bibitem{McLean:2010}
{\sc W.~McLean}, {\em Regularity of solutions to a time-fractional diffusion
  equation}, ANZIAM J., 52 (2010), pp.~123--138.

\bibitem{McLeanMustapha:2015}
{\sc W.~McLean and K.~Mustapha}, {\em Time-stepping error bounds for fractional
  diffusion problems with non-smooth initial data}, J. Comput. Phys., 293
  (2015), pp.~201--217.

\bibitem{MetzlerKlafter:2000}
{\sc R.~Metzler and J.~Klafter}, {\em The random walk's guide to anomalous
  diffusion: a fractional dynamics approach}, Phys. Rep., 339 (2000),
  pp.~1--77.

\bibitem{MustaphaAbdallahFurati:2014}
{\sc K.~Mustapha, B.~Abdallah, and K.~M. Furati}, {\em A discontinuous
  {P}etrov-{G}alerkin method for time-fractional diffusion equations}, SIAM J.
  Numer. Anal., 52 (2014), pp.~2512--2529.

\bibitem{MustaphaMustapha:2010}
{\sc K.~Mustapha and H.~Mustapha}, {\em A second-order accurate numerical
  method for a semilinear integro-differential equation with a weakly singular
  kernel}, IMA J. Numer. Anal., 30 (2010), pp.~555--578.

\bibitem{Nigmatulin:1986}
{\sc R.~R. Nigmatulin}, {\em The realization of the generalized transfer
  equation in a medium with fractal geometry}, Phys. Stat. Sol. B, 133 (1986),
  pp.~425--430.

\bibitem{SakamotoYamamoto:2011}
{\sc K.~Sakamoto and M.~Yamamoto}, {\em Initial value/boundary value problems
  for fractional diffusion-wave equations and applications to some inverse
  problems}, J. Math. Anal. Appl., 382 (2011), pp.~426--447.

\bibitem{SunWu:2006}
{\sc Z.-Z. Sun and X.~Wu}, {\em A fully discrete scheme for a diffusion wave
  system}, Appl. Numer. Math., 56 (2006), pp.~193--209.

\bibitem{Thomee:2006}
{\sc V.~Thom{\'e}e}, {\em Galerkin {F}inite {E}lement {M}ethods for {P}arabolic
  {P}roblems}, Springer-Verlag, Berlin, second~ed., 2006.

\bibitem{Yuste:2006}
{\sc S.~B. Yuste}, {\em Weighted average finite difference methods for
  fractional diffusion equations}, J. Comput. Phys., 216 (2006), pp.~264--274.

\bibitem{ZengLiLiuTurner:2015}
{\sc F.~Zeng, C.~Li, F.~Liu, and I.~Turner}, {\em Numerical algorithms for
  time-fractional subdiffusion equation with second-order accuracy}, SIAM J.
  Sci. Comput., 37 (2015), pp.~A55--A78.

\end{thebibliography}
\end{document}